\documentclass[11pt,a4paper,leqno]{amsart}
\usepackage{amsmath,stmaryrd}
\usepackage{amssymb}
\usepackage{amsthm}
\usepackage{epsf}
\usepackage{graphicx}
\usepackage{esint} 
\usepackage{dsfont}
\usepackage[pagebackref]{hyperref} 
\hypersetup{pdfpagemode=FullScreen,  colorlinks=true} 
\usepackage{enumerate} 
\usepackage{xcolor,bbm,easybmat,bm}
\usepackage{tikz-cd}
\usepackage{centernot}
\usepackage{comment}

\numberwithin{equation}{section}

\newcommand\footnoteref[1]{\protected@xdef\@thefnmark{\ref{#1}}\@footnotemark}
\makeatother

\newcommand{\vp}{\varphi}
\newcommand{\dr}{\partial}

\DeclareMathOperator{\Span}{Span}
\DeclareMathOperator{\diver}{div}
\DeclareMathOperator{\sgn}{sgn}
\DeclareMathOperator{\dist}{dist}

\DeclareMathOperator{\diam}{diam}

\DeclareMathOperator{\Tr}{Tr}

\DeclareMathOperator{\CMs}{CM_{sup}}
\DeclareMathOperator{\CM}{CM}

\DeclareMathOperator{\Cone}{Cone}

\newcommand{\1}{{\mathds 1}}

\newcommand{\ms}{\medskip}

\newcommand{\R}{\mathbb R}
\newcommand{\bN}{\mathbb N}

\newcommand{\bp}{\noindent {\em Proof: }}
\newcommand{\ep}{\hfill $\square$ \medskip}

\newcommand{\wt}{\widetilde}

\newcommand{\Z}{\mathcal Z}

\newcommand{\Rn}{\mathbb R^n}
\newcommand{\norm}[1]{\left\Vert#1\right\Vert}
\newcommand{\abs}[1]{\left\vert#1\right\vert}
\newcommand{\br}[1]{\left(#1\right)}
\newcommand{\set}[1]{\left\{#1\right\}}

\newcommand{\om}{\Omega}
\newcommand{\pom}{\partial\Omega}

\newcommand{\dint}{\int\!\!\!\!\!\int}
\def\Yint#1{\mathchoice
	{\YYint\displaystyle\textstyle{#1}}%
	{\YYint\textstyle\scriptstyle{#1}}%
	{\YYint\scriptstyle\scriptscriptstyle{#1}}%
	{\YYint\scriptscriptstyle\scriptscriptstyle{#1}}%
	\!\dint}
\def\YYint#1#2#3{{\setbox0=\hbox{$#1{#2#3}{\iint}$}
		\vcenter{\hbox{$#2#3$}}\kern-.51\wd0}}
\def\longdash{\mkern-1.5mu{-}\mkern-7.5mu{-}} 

\def\fiint{\Yint\longdash}

\theoremstyle{plain}
\newtheorem{theorem}[equation]{Theorem}
\newtheorem{lemma}[equation]{Lemma}
\newtheorem{corollary}[equation]{Corollary}
\newtheorem{proposition}[equation]{Proposition}

\newtheorem{question}[equation]{Question}
\newtheorem{definition}[equation]{Definition}

\theoremstyle{definition}

\theoremstyle{remark}
\newtheorem{remark}[equation]{Remark}

\begin{document}
\title[Stability of $L^p$ Dirichlet problem]{Stability of $L^p$ Dirichlet problem under small bi-Lipschitz transformations of domains}

\author[Feneuil]{Joseph Feneuil}
\address{Joseph Feneuil. Laboratoire de math\'ematiques d'Orsay, Universit\'e Paris-Saclay, France}
\email{joseph.feneuil@universite-paris-saclay.fr}

\author[Li]{Linhan Li} 
\address{Linhan Li. School of Mathematics, The University of Edinburgh, Edinburgh, UK }
\email{linhan.li@ed.ac.uk}

\author[Zhuge]{Jinping Zhuge}
\address{Jinping Zhuge. Morningside Center of Mathematics, Academy of Mathematics and systems science, Chinese Academy of Sciences, Beijing 100190, China}
\email{jpzhuge@amss.ac.cn}
\thanks{J. Zhuge is partially supported by NNSF of China (No. 12494541, 12288201, 12471115).
}

\maketitle

\begin{abstract} 
We show that small bi-Lipschitz deformations of a Lipschitz domain (with possibly large Lipschitz constant) preserve the solvability of the Dirichlet problem for the Laplacian with boundary data in $L^p$, for the same value of $p>1$. As a consequence, for {\it all} $p\in(1,\infty)$, we obtain the solvability of the $L^p$ Dirichlet problem for small Lipschitz perturbations of convex domains, thereby unifying two fundamentally different settings in which such results were previously known: convex and $C^1$ domains. The key ingredient and novelty of our approach is a  construction of a change of variables based on a {\it non-constant} basis derived from the Green function, which encodes the geometry of the base domain.
\end{abstract}

\ms\noindent{\bf Keywords: } $L^p$ Dirichlet problem, Laplacian, bi-Lipschitz change of variables, perturbations, Green function, Lipschitz domains, convex domains, quasiconvex domains.

\ms\noindent

\tableofcontents

\section{Introduction}

\subsection{Motivation and statement of main results} 
We consider the Dirichlet problem for the Laplacian in a domain $\om\subset\Rn$. 
The solvability of the Dirichlet problem with boundary data in $L^p$---or $(D)_p$ for simplicity, see Definition~\ref{def.Dp}---is closely tied to the geometry of the underlying domain and its boundary. A substantial body of recent work (see e.g. \cite{HM14,HMU14,AHMMT20}) has culminated in identifying the necessary and sufficient geometric conditions for solving the $L^p$ Dirichlet problem for {\it some} (large) $p>1$. Roughly speaking, the right geometric condition is some weak connectivity of the domain in addition to uniform rectifiability of the boundary of the domain, the latter being a strictly more general notion than a union of Lipschitz graphs.  
In contrast, the right geometric notion for the solvability of the $L^p$ Dirichlet problem for {\it all} $p\in (1,\infty)$ is far from clear. It remains an open question:
\begin{question}\label{Q1}
    What are the necessary and sufficient conditions on the domain for solving the $L^p$ Dirichlet problem $(D)_p$ for all $p>1$ ?
\end{question}

Noticeably, there are two fundamentally different types of domains on which the $L^p$ Dirichlet problem for the Laplacian is solvable for all $p\in (1,\infty)$: convex domains, and domains with ``very flat'' boundaries. Let us be more precise. 

For convex domains, as one has a pointwise bound on the gradient of the Green function up to the boundary, one can obtain $(D)_p$ for the full range of $p$ using the Hardy-Littlewood maximal function (this is observed in \cite{Shen06}). Combining this idea with Schauder theory, we can extend the result slightly beyond convex domains, namely, to domains whose boundary is (locally) the image of a convex graph under a $C^{1,\alpha}$ diffeomorphism (see Appendix~\ref{S.App}). However, since the argument relies heavily on boundary $C^1$ estimates of solutions, reducing the regularity of the $C^{1,\alpha}$ diffeomorphism to only $C^1$ is beyond reach of this method.

On the other hand, it is known that for any $p>1$, the $L^p$ Dirichlet problem is solvable in Lipschitz domains with sufficiently small Lipschitz constant, and in particular, in bounded $C^1$ domains. This follows from the result of \cite{JK82vmo}, which shows that on $C^1$ domains, the logarithm of the Poisson kernel ($\log k$) has vanishing mean oscillation (VMO) on the boundary. Alternatively, one can deduce this result from \cite{DPP07}, which treats a more general class of elliptic operators.  Also in this direction, \cite{KT97} generalizes the result to a broader class of domains whose boundaries are not necessarily given (locally) by graphs of functions. Specifically, Kenig and Toro showed that on {\it vanishing chord arc domains}---which can be viewed as a generalization of Lipschitz graph domains with vanishing Lipschitz constant---the logarithm of the Poisson kernel is in $\rm{VMO}(\pom)$, which implies that the $L^p$ Dirichlet problem for the Laplacian is solvable for all $p\in(1,\infty)$. Remarkably, there is also a converse to this result (\cite{KT03}): assuming that the domain is a chord-arc domain with sufficiently small constant, if $\log k\in \rm{VMO}(\pom)$, then the domain is a vanishing chord arc domain.
\smallskip

We can see that ``flatness'' of the boundary---such as the vanishing chord-arc condition---{\it cannot} be the right geometric criterion characterizing the solvability of $(D)_p$ for all $p>1$.
Indeed, convex domains can be Lipschitz domains with very large Lipschitz constants, and since the Lipschitz constant is invariant under dilation, the boundaries of such domains can remain far from flat at every scale---yet the $L^p$ Dirichlet problem is solvable for all $p>1$.  In particular, this means that neither convex nor vanishing chord-arc is a necessary condition for solving the $L^p$ Dirichlet problem for all $p>1$. The first step to understand the right geometric condition would thus require unifying these two conditions.

\smallskip

One of our contributions in this paper is 
identifying a class of domains which includes both convex domains and Lipschitz domains with vanishing Lipschitz constants, and on which the $L^p$ Dirichlet problem is solvable for all $p>1$ (see Corollary~\ref{MainCor}). Notably, this class of domains is invariant under any $C^1$ diffeomorphism. We postpone the precise definition and statement until the end of this section. 

Our approach is to study the following question, which can be viewed as the first step toward understanding Question~\ref{Q1}.
\begin{question}\label{Q2}
    Let $p\in(1,\infty)$. What kinds of domain perturbations preserve $(D)_p$ for the same value of $p$?
\end{question}
While there are many perturbation results for {\it operators} that preserve the solvability of the $L^p$ Dirichlet problem---see e.g. \cite{Dah86,Esc96,DPP07,MPT13,MPT14small,CHM,AHMT23,AAAHK11,BHLMP24}, see also perturbation results for operators that do not preserve the same value of $p$ \cite{FKP91,CHMT20,DSU,FP23,Hofdrift}---to our knowledge there has been no perturbation theory concerning {\it domains}, or at least, it has not usually been viewed in this way. On the other hand, one can apply certain perturbation results for operators to obtain corresponding perturbation results for domains, which is likely one of the original motivations for studying perturbation theory for operators. However, even when restricted to small Lipschitz perturbations in the transversal direction of a Lipschitz graph domain, existing theory suffers from one of the following 3 limitations: 
\begin{itemize}
    \item it fails to preserve the same value of 
$p$,
\item it does not yield results in the range $1<p<2$, 
or
\item it requires the base domain to have a sufficiently flat boundary (such as Lipschitz domain with small constant).
\end{itemize} 
We refer the reader to Section~\ref{S.old} for a detailed discussion on this.
\medskip

Our main result gives an answer to Question~\ref{Q2}, overcomes all the limitations of previous methods, and includes a substantially broader class of perturbations. In fact, our result applies to any bi-Lipschitz transformation sufficiently close to the identity. Moreover, we show that these transformations of domains preserve $(D)_p$ for the same value of $p$ for any $p\in(1,\infty)$, and we do not require any smallness condition on the Lipschitz constant of the base domain.
Specifically, our main theorem is the following.
\begin{theorem} \label{MainTh}
Let $\Omega_0\subset\Rn$ be a Lipschitz graph domain, that is, there exists $M>0$ and a $M$-Lipschitz function $g:\R^{n-1}$ such that \[\Omega_0:= \{ (x,t) \in \R^{n-1}\times\R,\, t> g(x)\}.\]
Let $p\in (1,\infty)$ be such that the $L^p$ Dirichlet problem  for the Laplacian is solvable in $\Omega_0$.

There exists $\epsilon_0>0$ such that if $\Phi :\, \R^n \to \R^n$ is a bi-Lipschitz map whose Jacobian matrix $\nabla\Phi$ satisfies $\|\nabla\Phi - I\|_\infty \leq \epsilon_0$, then the $L^p$ Dirichlet problem for the Laplacian is solvable in $\om = \Phi(\Omega_0)$. 

Moreover, $\epsilon_0 = \epsilon_0(n,p,C_p,M)$ depends only on the dimension $n$, the parameter $p$, the reverse H\"older constant $C_p$ in \eqref{defRH}, and the Lipschitz constant $M$.
\end{theorem}
Some remarks are in order.
\begin{remark}\label{rm.Om-graph}
    Under the assumption of Theorem \ref{MainTh}, the perturbed domain $\om=\Phi(\om_0)$ takes the form of \[\om=\set{(x,t)\in\R^{n-1}\times\R: t>g(x + \psi(x)) + h(x)},\] where $\psi:\R^{n-1}\to\R^{n-1}$ and $h:\R^{n-1}\to\R$  are mappings that satisfy $|\nabla \psi| + |\nabla h| \lesssim \epsilon_0$. We provide a proof of this observation in Appendix~\ref{S.App-rmk}.
 \end{remark}
 \begin{remark}\label{rm.Om-transverLip}
     In view of Remark~\ref{rm.Om-graph}, the class of allowable domain perturbations includes, in particular, small Lipschitz perturbations in the transversal direction. More precisely, our result applies to any perturbation of $\om_0$ of the form
     \begin{equation}\label{eq.om:g+h}
         \om=\set{(x,t)\in\R^{n-1}\times\R: t>g(x)+h(x)},
     \end{equation}
     where $h:\R^{n-1}\to\R$ is any $\epsilon_0$-Lipschitz function. On the other hand, the class of perturbations covered by our theorem is strictly more general than the class described by \eqref{eq.om:g+h}, as the function $g(x + \psi(x)) - g(x)$ may not have small Lipschitz constant even if $\psi$ does.
 \end{remark}
 \begin{remark}
     When restricted to $p \ge 2$, Theorem~\ref{MainTh} can be derived from known arguments and results. In fact, since both $\om_0$ and the perturbed domain $\om$ are Lipschitz domains, the result can be obtained by using the solvability of the $L^p$ Dirichlet problem for the Laplacian in Lipschitz domains for all $p \in [2,\infty)$, which is established in \cite{Dah77} and \cite{Ver84}.  However, none of the available methods apply to the range $1<p<2$---see further discussion in Section~\ref{S.old}. We take a new approach which, by contrast, works uniformly for every $p>1$ and does not distinguish between the cases $p \ge 2$ and $p<2$.
 \end{remark}

\medskip 
Equipped with Theorem~\ref{MainTh}, we can now introduce the class of domains that unifies  convex and $C^1$ domains. Roughly speaking, these domains are locally the image of a convex domain under a bi-Lipschitz map close to the identity.   
Specifically, we define the following.
\begin{definition} \label{defquasiconvex}
We say that a bounded domain $\Omega\subset\Rn$ is {\bf strongly quasiconvex} if there exists $M>0$ such that for any $\epsilon_0>0$, there exists $r_0>0$ such that for any $x_0\in \partial \Omega$, there exists an coordinate basis $\{e_1,\dots,e_n\}$, an $M$-Lipschitz convex function $g_{x_0}$ and a bi-Lipschitz map $\Phi_{x_0}:\, \R^n \to \R^n$ whose Jacobian matrix $\nabla\Phi_{x_0}$ satisfies $\|\nabla\Phi_{x_0} - I\|_\infty \leq \epsilon_0$ such that 
\[\Omega \cap B(x_0,r_0) = \Phi_{x_0}\Big( \{ (x,t) \in \R^{n-1}\times\R,\, t>g_{x_0}(x)\} \Big) \cap B(x_0,r_0).\]
\end{definition}
It is obvious from definition that any convex domain and any Lipschitz domain with vanishing Lipschitz constant (or any bounded $C^1$ domain) are strongly quasiconvex. We refer to this class of domains as  {\it strongly quasiconvex} because the condition resembles that of {\it quasiconvex domains} in the literature (see, for instance, \cite{JLW10,Z20,ZZ23}).
\begin{remark}\label{rm.C1invariant}
     The class of strongly quasiconvex domains is stable under any $C^1$ diffeomorphism (not necessarily with small constants). We give a sketch of the proof of this claim in Appendix~\ref{S.App-rmk}.
     In particular, a $C^1$ diffeomorphism transforms a convex domain into a strongly quasiconvex domain. This suggests that the class of strongly quasiconvex domains is a natural setting to consider.
\end{remark}

By an appropriate localization argument (see Section~\ref{S.pfCor}), one can deduce the our main corollary from Theorem~\ref{MainTh}.
\begin{corollary} \label{MainCor}
Let $\Omega \subset \R^n$ be a  bounded strongly quasiconvex domain. Then the $L^p$ Dirichlet problem for $-\Delta$ in $\Omega$ is solvable for all $p\in (1,\infty)$.
\end{corollary}

By the duality between the Dirichlet problem and the Regularity problem (see Definition~\ref{def.Rp}) in Lipschitz domains \cite{Shen07}, the $L^p$ Regularity problem is solvable whenever the $L^{p'}$ Dirichlet problem is solvable and the $L^q$ Regularity problem is solvable for some $q>1$. Since the latter is known to hold for the Laplacian in Lipschitz domains (\cite{Ver84, DK87}), we obtain the same range of $p$ for which the $L^p$ Regularity problem is solvable as for the Dirichlet problem.
\begin{corollary}
    Let $\Omega \subset \R^n$ be a  bounded strongly quasiconvex domain. Then the $L^p$ Regularity problem for $-\Delta$  in $\om$ is solvable for all  $p\in (1,\infty)$.
\end{corollary}

\subsection{Existing methods and their limitations}\label{S.old}
If one views Lipschitz domains with sufficiently small constant as a small Lipschitz perturbation of flat sets, and vanishing chord arc domains as a small perturbation of Lipschitz domains with vanishing constant, then the solvability results mentioned at the beginning of the previous section on $(D)_p$ for all $p>1$ on these two types of domains can be interpreted as perturbation results for domains that preserve the solvability of the $L^p$ Dirichlet problem, for the same value of $p>1$. One feature common to these two results is that both involve small perturbations from ``flat sets''. 
\medskip

Since our base domain may have very large Lipschitz constant---and thus its boundary far from being flat---existing methods either fail to preserve the same value of $p$, or they do not produce the result in the range $1<p<2$. In fact, in most cases, it is not even clear whether any of the existing methods could be applied to the general perturbations considered in our Theorem~\ref{MainTh}. Therefore, we shall often restrict our comparisons to a special case of our result: when the small Lipschitz perturbation occurs only in directions transversal to the boundary of the base domain, i.e., when the perturbed domain $\om$ has the form \eqref{eq.om:g+h}.
\smallskip

One possible approach to study stability of solvability of $L^p$ Dirichlet problem under domain perturbation is via the method of layer potentials, which requires studying the stability of the invertibility of layer potentials in $L^p$. This, in turn, requires not only the sovability of the $L^p$ Dirichlet problem in the base domain $\om_0$, but also that in its complement $\om_0^{\complement}$, together with the solvability of the $L^{p'}$ Regularity and Neumann problems in both $\om_0$ and $\om_0^{\complement}$. 
Such assumptions are too strong for our purposes: they cannot be used to deduce our Theorem~\ref{MainTh} (even in the special case when only transversal perturbation occurs) for $1<p<2$, since knowing that $(D)_p$ is solvable in $\om_0$ alone yields neither $(D)_p$ in $\om_0^\complement$ nor solvability of the $L^{p'}$ Neumann problem in $\om_0$ or $\om_0^\complement$. For example, in a convex domain $\om_0$ the Dirichlet and Neumann problems are solvable for the full range of $p\in(1,\infty)$; but $\om_0^\complement$ is not convex and the Dirichlet and Neumann problems are generally not solvable for the full range of $p$. Note that when $p\ge 2$, the above assumptions are automatically satisfied, since both $\om_0$ and $\om_0^\complement$ are Lipschitz domains. Indeed, \cite{DK87} shows that  for the Laplacian, the $L^p$ Dirichlet, the $L^{p'}$ Regularity and Neumann problems are solvable in Lipschitz domains for all $p\in (2-\epsilon,\infty)$ for some small $\epsilon>0$ depending on the Lipschitz domain. Moreover, this range is optimal, in the sense that there are counterexamples for every $p<2$ (see \cite{KenigBJ}). Thus the method of layer potentials reveals a sharp dichotomy between the cases $p\ge 2$ and $1<p<2$: it gives no new information in the former range, and it breaks down in the latter.
\medskip

Alternatively, one might wish to relate perturbations of domains to corresponding perturbations of operators. However, the existing perturbation theory for elliptic operators does not apply to our setting via any standard change of variables. Let us be more precise.

A standard approach to treating a boundary value problem on a Lipschitz graph domain \[\om=\set{(x,t)\in\R^{n-1}\times\R: t>\vp(x)}\] is to transform it into a boundary value problem on the upper half-space $\Rn_+$ for a different operator via a bi-Lipschitz change of variables (see Proposition~\ref{prop.DbiLip} for a precise formulation of such an example). There are two types of elliptic operators in the form $L=-\diver A\nabla$ arsing from two different bi-Lipschitz change of variables:
\begin{enumerate}
    \item ``$t$-independent'' operators, which are operators whose coefficients are independent of the transversal variable, that is, $A(x,t)=A(x)$;
    \item ``DKP'' operators, which are operators whose coefficients satisfy a Carleson condition, namely, $\dist(\cdot,\pom)|\nabla A|\in\CMs$ (see Definition~\ref{def.CM}).
\end{enumerate}
The ``$t$-independent'' operators are associated with the simple flattening (or lifting) map:
\[
\rho_0: \Rn_+\to \om, \quad (x,t)\mapsto (x, t+\vp(x)),
\]
and the ``DKP'' operators are associated with the Nečas-Kenig-Stein map (see \cite{KP01}):
\[
\rho_1: \Rn_+\to \om, \quad (x,t)\mapsto (x, ct+\eta_t*\vp(x)),
\]
where $\eta_\epsilon$ is a standard mollifier, and $c$ is a constant chosen to make sure that the map is 1-1. 
\smallskip

For the map $\rho_0$, the Lipschitz constant of the boundary is reflected as the $L^\infty$ norm of the pull-back operator, and so Theorem~\ref{MainTh} in the special case  when the perturbed domain $\om$ has the form \eqref{eq.om:g+h} can be reduced to the stability of the solvability of the $L^p$ Dirichlet problem for symmetric, real, $t$-independent operators $L_0=-\diver A_0\nabla$ and $L_1=-\diver A_1\nabla $ in $\Rn_+$, under the perturbation 
\[
\norm{A_0-A_1}_{L^\infty}\le\epsilon.
\]
This question is studied for $p=2$ in \cite{AAAHK11} and for general $p$ in \cite{HMM15} via the method of layer potentials. One of the remarkable features of both works is that their theory applies to operators with complex coefficients. When restricted to the real-coefficient setting, however, the method faces the same limitations discussed above for layer potentials. In particular, it cannot be applied, at least directly, to obtain the range $1<p<2$ in our Theorem~\ref{MainTh}.   
\smallskip

We now discuss the map $\rho_1$. From its construction, one can check that the constant in the Carleson condition of the pull-back operator increases as the Lipschitz constant of $\vp$ increases. Keing and Pipher (\cite{KP01}) showed that the $L^p$ Dirichlet problem is solvable for some (typically large) $p>1$ for DKP operators (with possibly large Carleson constant) on $\Rn_+$, and hence also on Lipschitz domains, since the DKP condition is stable under bi-Lipschitz changes of variables. 
Later, \cite{DPP07} showed that for any $p>1$, if the Carleson constant of the coefficients of a DKP operator is {\it sufficiently small}, then $(D)_p$ is solvable on $\Rn_+$, and therefore on Lipschitz domains with {\it sufficiently small} Lipschitz constant. Both works employ clever and delicate integration-by-parts arguments that depend crucially on a distinguished variable $t$; as a result, one is inevitably led to the flat case $\Rn_+$. If we apply $\rho_1$ to the two Lipschitz domains $\om_0$ and $\om$, we get two DKP operators with {\it large} Carleson constants, for which we only know $(D)_p$ for some large $p$. Although their Carleson constants are close to each other in the special case when $\om$ is a small Lipschitz perturbation of $\om_0$ in the transversal direction, they are not the appropriate Carleson perturbations to each other. In particular, the existing perturbation theory on operators do not apply in this setting.

\subsection{Main ideas of our approach}\label{S.Novel}

The takeaway from the previous discussions is that we do {\it not} want to reduce a Lipschitz domain to the flat case $\Rn_+$ because the small perturbation of the domain is lost in the flattening process. The central idea---and the key novelty---of our approach is to construct a map $\rho$ directly {\it from the Lipschitz graph domain} 
\[\om_0=\set{(x,t)\in\R^{n-1}\times\R: t>g(x)}\]
to the perturbed domain 
$\om=\Phi(\om_0)$,
and to design $\rho$ in such a way that the small perturbation of the domains is appropriately captured in the pull-back operator $L_\rho$ so that it enjoys the same solvablity of the $L^p$ Dirichlet problem as the Laplacian on $\om_0$. 
\smallskip

The first challenge that we face here is that we lose the special variable $t$ (or equivalently, the special direction $e_n$) when constructing a map directly from $\om_0$. Note that this issue is not due to the level of generality of the perturbation we consider: even when the perturbation occurs only in directions transversal to the boundary of $\om_0$, i.e., when the perturbed domain takes the form \eqref{eq.om:g+h},
one can no longer consider maps such as $\rho_0$ or $\rho_1$ mentioned in Section~\ref{S.old} by viewing points in $\om$ as the image of points in $\Rn_+$. 
Our idea is to use the Green function with pole at infinity (see Lemma~\ref{lem GrEm_infty} for the definition) for the Laplacian in the base domain---denoted by $G$---to capture the geometry of $\om_0$. This is not too surprising, as the zero set of the Green function is $\pom_0$, and the level sets of $G$ gives a family of surfaces that covers $\om_0$. We therefore consider the vector field $\nabla G$ on $\om_0$, and build an orthonormal basis $\set{v_1,\dots,v_n}$ for any point $X\in\om_0$, with $v_n=\frac{\nabla G}{|\nabla G|}$. To make sure that it is possible, we show in Lemma~\ref{lemSL=>H} that $\nabla G$ is non-degenerate on $\om_0$, and moreover, that $\frac{G(X)}{|\nabla G(X)|}$ behaves like the distance from $X$ to $\pom_0$. What is even better is that the gradient of $\frac{G}{|\nabla G|}$ is `almost parallel' to $v_n$ at every $X\in\om_0$, in the sense that their difference satisfies a Carleson condition--- see Theorem~\ref{ThGFE2}. Although the result in Theorem~\ref{ThGFE2} looks delicate, we remark that little additional work is required to obtain it as the difficult estimates for the second derivatives of the Green function have already been established in \cite{FL23} and \cite{DLM}, and can also be found implicitly in \cite{Azzam19}. This implies that we can think of the function 
\[t(X):=\frac{G(X)}{|\nabla G(X)|}\quad\text{for }X\in\om_0\]
as the transversal variable $t$ in $\Rn_+$, and hence we have an appropriate special direction for $\om_0$---although the direction depends on the location in $\om_0$.  
\smallskip

Next, we carefully design $\rho$ so that the pull-back operator $L_\rho=-\diver A_\rho\nabla$ from the Laplacian on $\om$ under the map $\rho$ is a small Carleson perturbation of an intermediate elliptic operator $L_0:=-\diver A_0\nabla$ on $\om_0$, and  that $L_0$ enjoys the same solvability of the $L^p$ Dirichlet problem as the Laplacian on $\om_0$. Since small Carleson perturbations of operators preserve solvability of the $L^p$ Dirichlet problem (see Theorem~\ref{ThCarlpert} for a precise statement), this process ensures that  $L_\rho$ has the same Dirichlet solvability as the Laplacian on $\om_0$, and thus the Laplacian has the same Dirichlet solvability on $\om_0$ and $\om$. The major work and challenge is to construct the intermediate operator $L_0$ so that it is close to both the Laplacian on $\om_0$ and the pull-back operator $L_\rho$. This construction---together with the construction of the map $\rho$---is carried out in Sections~\ref{S.v} and \ref{S.cov}, and the properties of $L_0$ is summarized in Lemma~\ref{lemLrho}.
\smallskip

One of the key properties that we want for the coefficient matrix $A_0$ of the intermediate operator $L_0$ is the following: we construct another orthonormal basis $\set{w_1,\dots,w_n}$ from $\set{v_1,\dots,v_n}$, so that {\it when written in these two bases},  the Jacobian matrix of $\rho$ is close---up to a Carleson perturbation---to a matrix in block form. It will imply that $A_\rho$ is close to $A_0$, which, {\it when written in the basis $\set{v_1,\dots,v_n}$}, is of the form
\[
\begin{bmatrix} * & \mathbf{0}^T \\ \mathbf{0} & 1 \end{bmatrix}, \quad \text{where } \mathbf{0} = 0_{\R^{n-1}} = (0,\dots,0).
\]
This implies that $A_0 v_n=v_n$, which, thanks to the direct relation between $v_n$ and the Green function, will be enough to show that $L_0$ and $-\Delta$ have the same $L^p$ Dirichlet solvability. We note that reducing matters to a block-form matrix or to a small perturbation of such a matrix, is not new in this area (see e.g. \cite{HKMP,Fen22cov,DFMDahlberg,DHP23}) and is often the core of the matter, though the idea is sometimes obscured by intricate integration-by-parts arguments. But again, in all previous work, such reductions rely crucially on the structure of $\Rn_+$ and are based on the Euclidean basis $\set{e_1,\dots,e_n}$ that is constant. 
\smallskip

Let us now discuss the main ideas behind the construction of the basis $\set{w_1,\dots,w_n}$ from $\set{v_1,\dots,v_n}$, as well as the map $\rho$. We want $\set{w_1,\dots,w_n}$ to be a small perturbation of $\set{v_1,\dots,v_n}$, capturing the perturbation caused by the bi-Lipschitz map $\Phi$. In particular, the vector $w_n$ 
should be an analog of $v_n$ for the perturbed domain $\om=\Phi(\om_0)$, in the sense that it should play the role as a normal vector to the new boundary $\pom$ that propagates into $\om$. This leads us to look at the perturbation of the boundary, which can be described by the vector-valued function $h:\R^{n-1}\to\Rn$ defined by 
$h(z):=\Phi(z,g(z))-(z,g(z))$ (see \eqref{defh}). However, this function is not smooth enough (it is Lipschitz) to allow us control the resulting Jacobian, and so we introduce a smoothing of $h$ for any point $X\in\om_0$, which is defined as $\lambda(X)$ (see \eqref{deflambda}). 
We then define the map $\rho:\om_0\to\om$ to be 
\[
\rho(X)=X-t(X)v_n(X)+\lambda(X) +a(X)t(X)w_n(X),
\]
or equivalently, as in \eqref{defrho}. Here, one should think of $X-t(X)v_n(X)$ as the projection of $X\in \om_0$ onto $\pom_0$, and so $X-t(X)v_n(X)+\lambda(X)$ is near the  boundary of $\om$ smoothed at scale $t(X)$. The map $\rho$ preserves the special direction of $\pom$ due to our construction of $w_n$. The role of $a(X)$ (defined in \eqref{defaX}) is essentially to make the $nn$-th entry of the matrix $A_0$ equal to 1, which simplifies the step comparing $L_0$ to $-\Delta$.

In the simpler case when the perturbed domain takes the form \eqref{eq.om:g+h}, experts familiar with \cite{DFMDahlberg} may wish to compare our construction of $\rho$ with the change of variables introduced there. In \cite{DFMDahlberg}, a modification of the Nečas–Kenig–Stein map $\rho_1$ is introduced to show that the $L^p$ Dirichlet problem is solvable for some $p>1$ for a class of degenerate elliptic operators in domains with low-dimensional boundaries. One of the key advantages of their modification, even in the co-dimension 1 case, is that it is better adapted to the geometry of the target domain, in that it tends to preserve the orthogonal direction. In this respect, our construction is similar in spirit, but many new ideas are required as the argument must now be carried out relative to a non-constant basis $\{v_1,\dots,v_n\}$, must also accommodate a perturbation of the domain, and we consider perturbations much more general than those occur only in directions transversal to the boundary of the base domain.

\medskip
We conclude the introduction by noting that our approach appears robust enough to apply to more general elliptic operators. For instance, we are hopeful that our results could be extended to operators satisfying the DKP condition, although substantial modifications would be required. Our results may also extend to more general domains whose boundaries are not necessarily given by graphs, following the ideas of \cite{KT97}. In a different direction, our method might be adapted to address Question~\ref{Q2} for the Neumann problem; however, due to the inherent difficulty of the Neumann problem, new ideas would be needed. Finally, a precise characterization of the necessary condition in Question~\ref{Q1} remains open and is likely to be highly challenging.

\subsection*{Acknowledgment} We would like to warmly thank Pascal Auscher, Steve Hofmann, and Zhongwei Shen for their help in understanding the literature on layer potentials.

\section{Preliminaries}

\subsection{Notations and definitions}
\begin{itemize}
   \item For a scalar function $f:\Rn\to\R$, we adopt the convention that $\nabla f$ is represented as a column vector, while a point $X \in \mathbb{R}^n$ is represented as a row vector. 
    \item The dot product of two vectors of the same dimension yields a scalar. We denote it either by the usual dot ``$\cdot$'' or by the inner product notation $\langle \cdot, \cdot \rangle$. 
    \item We use $|\cdot|$ to denote the norm of vectors or matrices.  
We use $\|\cdot\|$ to denote norms on function spaces.  
 When dealing with vector- or matrix-valued functions, $\|\cdot\|$ should be understood as $\||\cdot|\|$, i.e., we first take the pointwise vector/matrix norm and then the function norm.
 \item We denote by $I_{ij}$ the entries of the identity matrix, i.e.,  
\[
I_{ij} =
\begin{cases}
1, & \text{if } i=j, \\
0, & \text{if } i\neq j.
\end{cases}
\]
\item $B_{\mathbb{R}^{n-1}}(x,r)$ denotes the ball in $\mathbb{R}^{n-1}$ centered at $x \in \mathbb{R}^{n-1}$ with radius $r$. We sometimes omit the subscript when the ambient space is clear from the context.
\item $\delta(X)$ denotes $\dist(X,\pom)$ when the underlying domain $\om$ is clear from the context.
\item Given a domain $\om$ and a point $X\in \om$, we denote $B_X:=B(X,\delta(X)/2)$, unless specified otherwise.
\item For $x\in\pom$, $r>0$, we denote by $\Delta(x,r)$ the boundary ball $B(x,r)\cap\pom$.
\item We denote by $\pi$ the orthogonal projection on the first $(n-1)$-coordinates, i.e.
\[\pi: \R^n\to\R^{n-1}, \, (x_1,\dots,x_{n-1},x_n) \mapsto (x_1,\dots,x_{n-1}).\]
\end{itemize}

We use the concept of Carleson measures throughout the article.
\begin{definition}[$\CMs$ and $\CM$]\label{def.CM}
    Let $\om_0$ be a Lipschitz domain in $\Rn$. We say a quantity $f$ defined on $\om_0$ satisfies the Carleson measure condition with sup (denoted by $f\in \CMs$) if $\sup_{\frac12B_X}|f|^2\frac{1}{\delta(X)}dX$ is the density of a Carleson measure on $\om_0$, that is, there is some constant $M>0$ such that 
    \[
    \iint_{B(x,r)\cap\om_0}\br{\sup_{Z\in \frac12B_Y}|f(Z)|}^2\frac{dY}{\delta(Y)}\le Mr^{n-1} \quad\text{for all }x\in\pom, r>0\footnote{ or $0<r<\diam(\pom_0)$ if $\om_0$ is bounded.}.
    \]
 If we want to emphasize the constant $M$, we write $f\in\CMs(M)$. 

We say a quantity $f$ defined on $\om_0$ satisfies the Carleson measure condition (denoted by $f\in \CM$) if $|f(X)|^2\frac{1}{\delta(X)}dX$ is the density of a Carleson measure on $\om_0$. 
\end{definition}

\begin{remark}\label{re.CMsup}
    Observe that if $f\in \CMs(M)$, then there exists $C>0$ such that $\|f\|_{L^\infty(\Omega_0)}\le CM^{1/2}$ for $X\in\om_0$.
\end{remark}

\medskip

Next, we introduce the Green function and elliptic measures. 

For a uniformly elliptic operator $L=-\diver A\nabla$, one can construct the Green function $G_L$ on a Lipschitz domain (cf. \cite{GW82}, and \cite{HK07} for unbounded domains) as well as the elliptic measure $\omega_L$. Although these can be done in more general settings, we only give the definitions and state the related results in Lipschitz domains and for {\bf symmetric} uniformly elliptic operators in order to lighten the notations and avoid introducing extra geometry.

We start with the $L$-elliptic measure with a pole in $\Omega$. For $X \in \Omega$, there exists a unique Borel probability measure $\omega^X_L$
on $\pom$, such that for  $f \in C^\infty(\pom)$ if $\pom$ is bounded and $f \in C_c^\infty(\pom)$ if $\pom$ is unbounded, the function
\begin{equation}\label{eq.EM}
    u(X) = \int_{\pom} f(y) \, d\omega_L^X(y)
\end{equation}
is the unique weak solution to the Dirichlet problem $Lu=0$ in $\om$, $u|_{\pom}=f$
satisfying $u \in C(\overline{\Omega})$, and  if $\om$ is unbounded, additionally $u(X) \to 0$ as $|X| \to \infty$ in $\Omega$. We call 
$\omega^X_L$ the $L$-elliptic measure (or the elliptic measure corresponding to $L$) with pole at $X$.

We can associate to $L=-\diver A\nabla$ a unique Green function in $\Omega$, 
$G_L(X,Y): \Omega \times \Omega \setminus  \text{diag}(\Omega\times \om) \to \R$ 
with the following properties: 
$G_L(\cdot,Y)\in W^{1,2}_{\rm loc}(\Omega\setminus \{Y\})\cap C(\overline{\Omega}\setminus\{Y\})$, 
$G_L(\cdot,Y)\big|_{\pom}\equiv 0$ for any $Y\in\Omega$, 
and $L G_L(\cdot,Y)=\delta_Y$ in the weak sense in $\Omega$, that is,
\begin{equation}\label{Greendef}
    \iint_\Omega A(X)\,\nabla_X G_{L}(X,Y) \cdot\nabla\vp(X)\, dX=\vp(Y), \qquad\text{for any }\vp \in C_c^\infty(\Omega).
\end{equation}
In particular, $G_L(\cdot,Y)$ is a weak solution to $L G_L(\cdot,Y)=0$ in $\Omega\setminus\{Y\}$. 
Given $Y\in\om$, we write $G_L^Y=G_L(\cdot,Y)$ and call it the Green function with pole at $Y$.
\smallskip 

For unbounded domains, it is convenient to define {\it the Green function and elliptic measure with pole at infinity}.  
One can prove the following lemma as in \cite[Lemma 3.7, Corollary 3.2]{KT99}.
\begin{lemma}[Elliptic measure and Green function with pole at infinity]\label{lem GrEm_infty}
Let $\om$ be a Lipschitz graph domain and let $L=-\diver{A\nabla}$ be a symmetric elliptic operator on $\Omega$. 
Then for any fixed $Z_0\in\Omega$, there exists a unique function $G_L\in C(\overline{\om})$ such that 
\[
\begin{cases}
L G_L = 0 \quad \text{in } \Omega\\
G_L>0 \quad \text{in } \Omega\\
G_L=0 \quad \text{on }\pom,\\
\end{cases}
\]
and $G_L(Z_0)=1$. 
In addition, for any sequence $\set{X_k}_k$ of points in $\Omega$
such that $\abs{X_k}\to \infty$, there exists a subsequence (which we still denote by 
$\set{X_k}_k$) such that if we set 
\[u_k(Y) := \frac{G_L(Y,X_k)}{G_L(Z_0, X_k)},\]
 then 
\begin{equation}\label{Greeninfty_lmt}
    u_k \text{ converges uniformly to } G_L \qquad\text{in any compact sets in }\overline{\Omega} \text{ as } k\to\infty.
\end{equation} 
Moreover, there exists a unique locally finite positive Borel 
measure $\omega_L$ on $\pom$ such that the Riesz formula 
\begin{equation}\label{Rieszinfty}
\int_{\pom} f(y) \, d\omega_L(y) = - \iint A(Y)\nabla G_L(Y) \cdot \nabla F(Y) \, dY 
\end{equation}
holds whenever $f \in C_c^\infty(\pom)$ and $F \in C_c^\infty(\Omega)$ are such that $F|_{\pom}=f$.

We call $\omega_L$ the elliptic measure with pole at infinity (normalized at $Z_0$), and $G_L$ the Green function with pole at infinity (normalized at $Z_0$). 
\end{lemma}

\begin{definition}[$(D)_{p}^L$]\label{def.Dp}
    Let $p\in (1,\infty)$. We say that the $L^{p}$ Dirichlet problem $(D)_p$ - or $(D)_p^L$ when we mention the operator - is solvable in $\om$ if there is $C>0$ such that for any $f\in C^\infty_0(\R^n)$, the solution $u_f$ to $Lu_f=0$ in $\om$ defined with the help of the elliptic measure by
\begin{equation} \label{defufhm}
u_f(X) := \int_{\partial \Omega} f(y)\, d\omega_L^X(y)
\end{equation}
satisfies
\[\|N(u_f)\|_{L^{p}(\partial \Omega)} \leq C \|f\|_{L^{p}(\partial \Omega)},\]
where 
\begin{equation*}
N(u)(x)=N_a(u)(x):=\sup_{Y\in\gamma(x)}|u(Y)|, \text{ and } \gamma(x)=\gamma_a(x)=\set{Y\in\om:\, |Y-x|<a\dist(Y,\pom)}
\end{equation*}
is a ``cone'' with vertex $x$ and aperture $a$. Note that the constant $C$ depends on the aperture $a$ of the cone. However, the $L^p$ solvability of the Dirichlet problem is independent of the aperture $a$, since real variable techniques shows that $\|N_a(.)\|_{L^p(\partial \Omega)}$ and $\|N_b(.)\|_{L^p(\partial \Omega)}$ are equivalent norms (see \cite{Stein93} in the Euclidean case, or Lemma 2.1 in \cite{MPT} for the proof in our setting). 
\end{definition}

\begin{definition}[$(R)_{p}^L$]\label{def.Rp}
    Let $p\in (1,\infty)$ and let $\om\subset\Rn$ be a Lipschitz domain. We say that the $L^{p}$ Regularity problem $(R)_p$ - or $(R)_p^L$ when we mention the operator - is solvable in $\om$ if there is $C>0$ such that for any $f\in C^\infty_0(\R^n)$, the solution $u_f$ to $Lu_f=0$ in $\om$ defined as in \eqref{defufhm}
satisfies
\[\|\wt N(\nabla u_f)\|_{L^{p}(\partial \Omega)} \leq C \|\nabla f\|_{L^{p}(\partial \Omega)},\]
where 
\begin{equation*}
\wt N(u)(x)= \sup_{Y\in\gamma (x)}\fiint_{B_Y}|u(Z)|dZ.
\end{equation*} 
\end{definition}

\subsection{Some properties on $(D)_p$}
It is well-known that the elliptic measure and Green function with pole at infinity can be compared via the comparison principle (see, for example, \cite{Ken94}). This is also true when the pole is at infinity.  As before, we state the results only in the specific settings relevant to our work, rather than in their full generality. 

\begin{lemma}[Lemmas 3.3 and 3.4 in \cite{DFMpert}]\label{lemG=om} 
Let $\Omega_0$ be a Lipschitz graph domain, and let $L = -\diver A \nabla$ be a {\bf symmetric} uniformly elliptic operator on $\Omega_0$. Then there exists $C>0$ such that the Green function with pole at infinity $G_L$ and the elliptic measure $\omega_L$ with pole at infinity verifies
\[C^{-1} \delta(X)^{n-2} G_L(X) \leq \omega_L(4B_X) \leq \omega_L(8B_X) \leq C \delta(X)^{n-2} G_L(X) \qquad \text{ for all } X \in \Omega_0,\]
where $B_X=B(X,\delta(X)/2)$. 
In particular, $\omega_L$ is a doubling measure. 

Moreover, the constant $C=C(n,M,C_{ellip})$ depends only the dimension $n$, the Lipschitz constant $M$ of $\Omega_0$, and the ellipticity constant $C_{ellip}$ of $L$.
\end{lemma}
There is also a finite-pole version of the lemma, which we omit here.

The elliptic measure and Green function with pole at infinity can be used to characterize the solvability of the Dirichlet problem.
\begin{theorem} \label{ThDP<=>RH}
Let $\Omega_0$ be a Lipschitz graph domain, and write $\sigma$ for the surface measure on $\partial \Omega_0$. Let $L=-\diver A\nabla$ be a symmetric uniformly elliptic operator, and $p\in (1,\infty)$. Then the following are equivalent:
\begin{enumerate}[(i)]
\item The $L^p$ Dirichlet problem $(D)_p^L$ is solvable in $\Omega_0$.
\item The elliptic measure with pole at infinity $\omega_L$ is absolutely continuous with respect to $\sigma$, and there exists $C_p>0$ such that the function $\kappa$ defined as
\[\kappa(x):= \limsup_{X\in \gamma(x)\atop X\to x} \frac{G_L(X)}{\delta(X)}\]
satisfies
\begin{equation} \label{defRH}
\left(\fint_{B(x,r)\cap \partial \Omega_0} \kappa^{p'}(y) \, d\sigma(y)\right)^\frac1{p'} \leq C_p \fint_{B(x,r)\cap \partial \Omega_0} \kappa(y) \, d\sigma(y) \qquad \text{ for all } x\in \partial \Omega_0, \, r>0.
\end{equation}
Here $p'$ stands for the H\"older conjugate of $p$.
\end{enumerate}
When $\Omega_0$ is a bounded Lipschitz domain, we take a point $X_0\in\om_0$ such that $\dist(X_0,\partial \Omega_0) \approx \diam (\Omega)$. Then (i) is equivalent to 
\begin{enumerate}[(i)]
    \setcounter{enumi}{1} 
    \item The elliptic measure $\omega_L^{X_0}$  is absolutely continuous with respect to $\sigma$, and there exists $C_p>0$ such that the function $\kappa^{X_0}$ defined as
\end{enumerate}
\[
\kappa^{X_0}(x):=\limsup_{X\in \gamma(x)\atop X\to x} \frac{G_L^{X_0}(X)}{\delta(X)},
\]
satisfies 
\begin{equation} \label{defRH'}
\left(\fint_{B(x,r)\cap \partial \Omega_0} \br{\kappa^{X_0}}^{p'}(y) \, d\sigma(y)\right)^\frac1{p'} \leq C_p \fint_{B(x,r)\cap \partial \Omega_0} \kappa^{X_0}(y) \, d\sigma(y) \qquad \text{for } x\in \partial \Omega, \, 1<r<\diam(\om).
\end{equation}
\end{theorem}

\bp
It is well-known that $(i)$ is equivalent to $\omega_L\ll \sigma$ and that there exists $C_p>0$ such that the function $k:=\frac{d\omega_L}{d\sigma}$ satisfies \eqref{defRH}, or, in the bounded case, $k^{X_0}:=\frac{d\omega_L^{X_0}}{d\sigma}$ satisfies \eqref{defRH'}---see, for instance, \cite[Theorem 1.7.3]{Ken94}. We claim that $k$ and $\kappa$ are comparable, and that $k^{X_0}$ and $\kappa^{X_0}$ are comparable, with the implicit constants depending only the dimension $n$, the Lipschitz constant $M$ of $\Omega_0$, and the ellipticity constant $C_{ellip}$ of $L$. Note that the equivalence follows once this claim is justified.

We only prove the unbounded (i.e. Lipschitz graph domain) case as the bounded case is similar. 
For any $X\in\gamma(x)$, write $s_X$ for $\delta(X)$. Then by Lemma~\ref{lemG=om},
\begin{equation}\label{eq.Gd=omsig}
    \frac{G_L(X)}{\delta(X)}\approx \frac{\omega_L(\Delta(x,s_X))}{\sigma(\Delta(x,s_X))}
\end{equation}
since $\sigma(\Delta(x,s_X))\approx s_X^{n-1}=\delta(X)^{n-1}$. This implies that 
\[
k(x)=\liminf_{s\to 0}\frac{\omega_L(\Delta(x,s))}{\sigma(\Delta(x,s))}\lesssim\limsup_{X\in \gamma(x)\atop X\to x}\frac{G_L(X)}{\delta(X)}\lesssim\limsup_{s\to 0}\frac{\omega_L(\Delta(x,s))}{\sigma(\Delta(x,s))}=k(x)\quad \sigma-\text{a.e. }x\in\pom_0 
\]
by the assumption that $\omega_L\ll\sigma$. Hence we have that $\kappa\approx k$ a.e..
\ep

The next result that we need concerns the stability of the Dirichlet problem under small Carleson perturbations. The earliest result of this kind was established by Dahlberg in \cite{Dah86}, which proves stability under Carleson perturbations with vanishing trace. The result below follows directly from \cite[Theorem 1.1]{CHM} together with Theorem \ref{ThDP<=>RH}. We note that the conclusion remains valid without assuming symmetry of the coefficient matrices; however, we state it in this form since it suffices for our application, and we are not aware of an appropriate reference in the general case.
\begin{theorem} \label{ThCarlpert}
Let $\Omega_0$ be a Lipschitz graph domain, and let $L_0:=-\diver A_0 \nabla$ and $L_1:=-\diver A_1\nabla$ be symmetric uniformly elliptic operators. 
Let $p\in (1,\infty)$ such that the Dirichlet problem for $L_0$ in $\Omega_0$ with boundary data in $L^p(\partial \Omega_0)$ is solvable .

There exists $\epsilon_1>0$ such that if $A_1-A_0 \in \CMs(\epsilon_0)$, then the Dirichlet problem for $L_1$ in $\Omega$ with data in $L^p(\partial \Omega_0)$ is solvable, meaning that the reverse H\"older bounds \eqref{defRH} holds with some constant $C'_{p}$.

\medskip

Moreover, the constants $\epsilon_1 = \epsilon_1(n,p,C_p,M,C_{ellip})$ and $C'_{p} = C'_{p}(n,p,C_p,M,C_{ellip})$ depends only on the dimension $n$, the parameter $p$, the reverse H\"older constant $C_p$ in \eqref{defRH} for the $L_0$ in $\Omega_0$, the Lipschitz constant $M$ of $\Omega_0$, and the ellipticity constant $C_{ellip}$ of $L_0$.
\end{theorem}

We also need the stability of the Dirichlet problem under bi-Lispchitz change of variables.
\begin{proposition}\label{prop.DbiLip}
    Let $\om_0$ and $\om$ be two Lipschitz domains in $\Rn$, and let $\rho: \om_0\to\om$ be a bi-Lipschitz map that satisfies $\rho(\pom_0)=\pom$. Then the solvability of the $L^p$ Dirichlet problem is stable under $\rho$. That is, let $L=-\diver A\nabla $ be an elliptic operator on $\om$, and let $L_\rho$ be the pull-back of $L$ under $\rho$ (a.k.a. the conjugate of $L$ by $\rho$), then the $L^p$ Dirichlet problem for $L$ is solvable in $\om$ if and only if the $L^p$ Dirichlet problem for $L_\rho$ is solvable in $\om_0$. 
\end{proposition}
\begin{proof}
   This result is folklore; we sketch its proof for completeness. We only show that if the $L^p$ Dirichlet problem is solvable for $L$ in $\om$, then the $L^p$ Dirichlet problem is solvable for $L_\rho$ in $\om_0$, as the other direction is similar. Suppose the Lipschitz constant of $\rho$ and $\rho^{-1}$ is bounded by $M$.
   Let $g\in L^p(\pom_0)$ be arbitrary, and let $f=g\circ \rho^{-1}$. Then $f\in L^p(\pom)$, and so by assumption, 
   \begin{equation}\label{eq.Nuf}
       \norm{N_b(u_f)}_{L^p(\pom)}\le C\norm{f}_{L^p(\pom)}
   \end{equation}
   for some $C>0$, where $u_f$ is defined as in \eqref{defufhm}, and $b>0$ is any aperture which will be determined later.  One can verify that the solution $v_g$ given by
   \[
   v_g(X):=\int_{\pom_0}g(y)d\omega_{L_\rho}^X(y)
   \]
   satisfies $v_g=u_f\circ \rho$. 
   Our goal is to show that $\norm{N(v_g)}_{L^p(\pom_0)}\le C'\norm{g}_{L^p(\pom_0)}$ for some $C'>0$. Observe that for any $x\in\pom_0$, we have 
   \[N(v_\rho)(x)=\sup_{Y\in\gamma(x)}|u_f(\rho(Y))|\le \sup_{Y\in\gamma_{M^2}(\rho(x))}|u_f(Y)|=N_{2M^2}(u_f)(\rho(x)),\]
   where $\gamma(x)=\set{Y\in\om_0:\, |Y-x|<2\dist(Y,\pom_0)}$ is a cone in $\om_0$ with vertex $x$ and aperture 1, and $\gamma_{2M^2}(\rho(x))$ is a cone in $\om$ with vertex at $\rho(x)$ and aperture $M^2$. 
  In fact,  for $Y\in\gamma(x)$, since $\rho$ is bi-Lipschitz with constant $M$, \[|\rho(Y)-\rho(x)|\le M|Y-x|\le 2M\dist(Y,\pom_0)\le 2M^2\dist(\rho(Y),\rho(\pom_0)),\]
  which implies that $\rho(Y)\in\gamma_{2M^2}(\rho(x))$. Therefore, taking $b=2M^2$ in \eqref{eq.Nuf}, we get that 
  \[
  \norm{N(v_\rho)}_{L^p(\pom_0)}\le\norm{N_{M^2}(u_f)\circ \rho}_{L^p(\pom_0)}\le C_M\norm{N_{M^2}(u_f)}_{L^p(\pom)}\le C_M\norm{f}_{L^p(\pom)}\le C_{M}\norm{g}_{L^p(\pom_0)},
  \]
  which proves that the $L^p$ Dirichlet problem is solvable for $L_\rho$ in $\om_0$.
\end{proof}

\subsection{Estimates for the Green function.}

\begin{lemma} \label{lemNG<G/d}
 Let $\om_0$ be a domain  in $\Rn$ and let $\Delta u=0$ in $\om_0$. Then there exists $C>0$ such that 
\[
|\nabla u(X)| \leq  \frac{C}{\delta(X)}\br{\fint_{B_X}|u(Y)|^2dY}^{1/2}  \qquad  \text{ for all } X\in \Omega_0,
\]
where  $B_X:= B(X,\delta(X)/2)$. If in addition, $u\ge 0$ in $\om_0$, then 
\[
|\nabla u(X)| \leq C u(X) / \delta(X) \qquad  \text{ for all } X\in \Omega_0.
\]
In particular, if $\Omega_0$ is unbounded and if $G$ is the Green function with pole at infinity for the Laplacian, as defined in Lemma~\ref{lem GrEm_infty}, then there exists $C>0$ such that 
\[|\nabla G(X)| \leq C G(X) / \delta(X) \qquad  \text{ for all } X\in \Omega_0.\]
\end{lemma}

\bp
This lemma is a classical consequence of the mean value property of harmonic functions, and also holds for solutions to more general elliptic PDEs. We include the short proof for completeness. Since $\partial_i u$ is harmonic in $\Omega_0$ for each $i\in \{1,\dots,n\}$, the De Giorgi-Nash-Moser estimate gives that
\[ |\nabla u(X)|^2 := \sum_{i=1}^n |\partial_i u(X)|^2 \lesssim \sum_{i=1}^n \fint_{\frac12 B_X} | \partial_i u|^2 dY = \fint_{\frac12 B_X} | \nabla u|^2 dY.\]
Using the Caccioppoli inequality, we further get
\[ |\nabla u(X)|^2 \lesssim \frac{1}{\delta(X)^2} \fint_{B_X} u^2(Y) dY,\]
which is the first claim. If $u\ge 0$ in $\om_0$, then the right-hand side above is further bounded by $ \frac{u(X)^2}{\delta(X)^2}$
by the Harnack inequality. The lemma follows.
\ep

\begin{theorem} \label{ThGFE}
Let $\Omega_0$ and $G$ be as in Lemma~\ref{lemNG<G/d}. Then there holds
\[\dfrac{\delta^2|\nabla^2 G|}{G} \in \CMs(M),\]
where $M$ depends on the Lipschitz character of $\Omega_0$.
\end{theorem}

\bp
This result is essentially a special case of the main result of \cite{FL23} (see alternatively \cite[Corollary 1.17]{DLM} and \cite{Azzam19}). Namely, \cite[Theorem 1.13]{FL23} gives that $\dfrac{\delta^2|\nabla^2 G|}{G} \in \CM(M)$, that is,
\begin{equation}\label{eq.D2GCM}
    \iint_{B(x,r)\cap\om_0}\abs{\frac{\nabla^2 G(Y)}{G(Y)}}^2\delta(Y)^3dY\le M r^{n-1} \quad\text{for all }x\in\pom_0,\, r>0.
\end{equation}
So we only need to show that we can take supreme in $\frac12B_Y$ in the integrand, where $B_Y=B(Y,\delta(Y)/2)$. Since $\dr_i\dr_j G$ is harmonic in $\om_0$ for each $i,j\in\set{1,\dots,n}$, the De Giorgi-Nash-Moser estimate gives that 
\[\sup_{Z\in \frac12B_Y}|\nabla^2 G(Z)|\lesssim \fiint_{\frac12 B_Y}|\nabla^2G(Z)|dZ.\] 
Therefore, applying Harnack inequality to $G$ in $\frac{1}{2}B_Y$ and H\"older's inequality, we get that
\begin{multline*}
    \iint_{B(x,r)\cap\om_0}\sup_{Z\in\frac12 B_Y}\abs{\frac{\delta(Z)^2\nabla^2 G(Z)}{G(Z)}}^2\frac{dY}{\delta(Y)}
    \lesssim \iint_{B(x,r)\cap\om_0}\fiint_{Z\in\frac12 B_Y}\frac{\abs{\nabla^2 G(Z)}^2}{G(Z)^2}\delta(Z)^3dZ dY\\
    =\iint_{Z\in B(x,2r)\cap\om_0}\frac{\abs{\nabla^2 G(Z)}^2}{G(Z)^2} \delta(Z)^3\br{\iint_{|Y-Z|<\delta(Y)/4}\delta(Y)^{-n}dY}dZ
\end{multline*}
by Fubini's theorem. Note that $|Y-Z|<\delta(Y)/4$ implies that $\frac45\delta(Z)\le\delta(Y)\le\frac43\delta(Z)$ and that $|Y-Z|\le \frac13\delta(Z)$. Hence, 
\[
\iint_{|Y-Z|<\delta(Y)/4}\delta(Y)^{-n}dY\lesssim \delta(Z)^{-n}\iint_{B(Z,\frac{\delta(Z)}{3})}dY\lesssim1,
\]
from which it follows that for all $x\in\pom_0$ and $r>0$,
\[
\iint_{B(x,r)\cap\om_0}\sup_{Z\in\frac12 B_Y}\abs{\frac{\delta(Z)^2\nabla^2 G(Z)}{G(Z)}}^2\frac{dY}{\delta(Y)}\lesssim \iint_{ B(x,2r)\cap\om_0}\frac{\abs{\nabla^2 G(Z)}}{G(Z)}^2\delta(Z)^3dZ\lesssim Mr^{n-1}
\]
by \eqref{eq.D2GCM}. The lemma follows. \ep

\section{The change of variables---the heart of the matter}

\subsection{Set-up estimates on Lipschitz graph domains}
Throughout this section, $\om_0$ is an unbounded domain in $\Rn$, and $G$ is the Green function for the Laplacian with pole at infinity in $\om_0$.

\begin{definition}
We say that the domain $\Omega_0$ satisfies the $\nabla G$'s non-degeneracy condition---called \eqref{H1}---if there exists $C_1>0$ such that
\begin{equation} \label{H1} \tag{H1}
G(X) \leq C_H \delta(X) |\nabla G(X)| \qquad \text{ for all } X\in \Omega_0.
\end{equation}
We say that the domain $\Omega$ satisfies the preferred direction condition---called \eqref{H2}---if there exists an orthonormal basis $\{e_1,\dots,e_n\}$ and a constant $C_2>0$ such that 
\begin{equation} \label{H2} \tag{H2}
\left<\nabla G(X),e_n \right> \geq (C_2)^{-1} |\nabla G(X)|\quad \text{ for } X\in \Omega_0.
\end{equation}
\end{definition}

\begin{lemma} \label{lemSL=>H}
Let $\Omega_0:=\{ (x,t) \in \R^n,\, t> g(x)\}$ be a Lipschitz graph domain. Then $\Omega_0$ satisfies \eqref{H1}--\eqref{H2}.
\end{lemma}

This lemma follows from \cite[Theorem 11.10]{CS.AMS} and Lemma~\ref{lemNG<G/d}. In fact, we can assume without loss of generality that $g(0)=0$. Given $X_0=(x_0,t_0)\in\om_0$, by translation we can assume that $X_0=(0,t_0)$. Let $\delta_0=\delta_0(n,M)>0$ be the constant found in \cite[Theorem 11.10]{CS.AMS}, where $M$ is the Lipschitz constant of $g$. If $t_0<4\delta_0M$, then applying \cite[Theorem 11.10]{CS.AMS}\footnote{There is a typo in the statement, namely, (11.14) should be $u_{x_n}(x)\sim u(x)/\dist(x,\pom)$.} directly to $G$ gives 
\begin{equation}\label{eq.drnG}
    \dr_n G(X_0)\approx \frac{G(X_0)}{\delta(X_0)}.
\end{equation} 
If $t_0>4\delta_0M$, one gets the same estimate by applying the theorem to the harmonic function $u(X):=G\br{\frac{t_0}{2\delta_0 M}X}$ and taking $X_0=(0,2\delta_0M)$.
Then \eqref{H1} follows from \eqref{eq.drnG} immediately, while \eqref{H2} follows from \eqref{eq.drnG} and Lemma~\ref{lemNG<G/d}---recall that we write $\partial_n G$ for $\left<\nabla G(X),e_n \right>$.
\smallskip

For the sake of completeness, we provide an alternative proof of Lemma~\ref{lemSL=>H} in Appendix~\ref{S.altpf}.

\begin{theorem} \label{ThGFE2}
Let $\Omega_0:= \{(x,t)\in \R^n, \, t>g(x)\}$ be a Lipschitz graph domain. Then the Green function $G$ for $-\Delta$ with pole at infinity defined in Lemma~\ref{lem GrEm_infty} verifies 
\[ \delta \nabla\left( \frac{\nabla G}{|\nabla G|}\right) \in \CMs(M),\]
and
\[\nabla\left( \frac{G}{|\nabla G|}\right) - \frac{\nabla G}{|\nabla G|} \in \CMs(M),\]
where $M$ depends on the Lipschitz constant of $g$.
\end{theorem}

\bp
This theorem is a simple consequence of Theorem \ref{ThGFE} and \eqref{H1}, the latter being given by Lemma \ref{lemSL=>H}. Indeed, \eqref{H1} guarantees that $|\nabla G|$ never becomes 0, meaning that we can differentiate $\nabla G/|\nabla G|$ without subtilities and we have
\[ \left|\delta \nabla \left(\frac{\nabla G}{|\nabla G|} \right)\right| \leq 2 \frac{\delta |\nabla^2 G|}{|\nabla G|} \lesssim \frac{\delta^2 |\nabla^2 G|}{G} \in \CMs\]
by \eqref{H1} and then Theorem \ref{ThGFE}. Similarly,
\[ \left| \nabla  \left(\frac{G}{|\nabla G|} \right) - \frac{\nabla G}{|\nabla G|}\right| \leq \frac{G|\nabla^2 G|}{|\nabla G|^2} \lesssim  \frac{\delta^2 |\nabla^2 G|}{G}  \in \CMs\]
again by \eqref{H1} and Theorem \ref{ThGFE}. The theorem follows.
\ep

We need now an orthonormal basis whose last element is $\nabla G/|\nabla G|$.
 
\begin{lemma} \label{lemdefvi}
Let $\Omega_0$ satisfy \eqref{H1}--\eqref{H2}. There exists $n$ vectors functions $v_1,\dots,v_n$ on $\Omega_0$ such that:
\begin{enumerate}[(i)]
\item For each $X\in \Omega_0$, the collection $\{v_1(X),\dots,v_n(X)\}$ is an orthonormal basis of $\R^n$.
\item The last vector function is $v_n:= \frac{(\nabla G)^T}{|\nabla G|}$, so in particular
\[v_i(X) \cdot \nabla G(X) = 0 \qquad \text{ for each } X\in \Omega \text{ and } i \in \{1,\dots,n-1\}.\]
\item There exists $C>0$ such that, for each $i\in \{1,\dots,n\}$ and $X\in\om_0$,
\[ |\nabla v_i(X)| \leq C |\nabla v_n(X)|.\]
\end{enumerate}
\end{lemma}

\bp
The result by itself is not surprising, since we can simply see it as the variant of the Gram-Schmidt process, and can probably already be found in the literature. For completeness, we sketch the argument.

Let $\{e_1,\dots,e_n\}$ be the Cartesian orthonormal basis from condition \eqref{H2}. For any fixed $X\in\om_0$, we define the vectors $v_i(X)$ by a decreasing induction in $i$ as follows. For $i\in \{n-1,\dots,1\}$, 
\begin{equation} \label{defvi}
\tilde v_i(X):= e_{i} - \sum_{i<k\le n} \langle e_{i}, v_k(X)\rangle v_k(X), \quad \text{ and } \quad  v_i(X) := \dfrac{\tilde v_i(X)}{\abs{\tilde v_i(X)}}.
\end{equation}
The condition \eqref{H2} guarantees that 
$\set{e_1,\dots,e_{n-1},v_n(X)}$ is linearly independent, and so the construction automatically returns an orthonormal basis. Moreover, \eqref{H2} says that $\left< v_n,e_n\right> \geq c$ for some $c\in(0,1)$ independent of $X$. By the Parseval's identity,
\begin{equation}\label{eq.ei.vn}
    \sum_{i=1}^{n-1} |\left< e_i, v_n(X) \right>|^2 \le 1-c^2.
\end{equation}
We want to prove that 
\begin{equation}\label{eq.tildvi-nd}
    |\tilde v_i(X)| \geq c\quad\text{for  }X\in\om_0,\, i\in\set{1,\dots,n-1}.
\end{equation}
 To this objective, note that by the Gram-Schmidt process, $|\tilde{v}_i|$ is the distance from $e_i$ to the linear space spanned by $\{e_{i+1}, \cdots, e_{n-1}, v_n \}$. Thus
\begin{equation*}
\begin{aligned}
    |\tilde{v}_i|^2 & = \inf_{\alpha_k \in \R, i+1\le k\le n} \left< e_i - \sum_{k=i+1}^{n-1} \alpha_k e_k - \alpha_n v_n,\ e_i - \sum_{k=i+1}^{n-1} \alpha_k e_k - \alpha_n v_n \right> \\
    & = \inf_{\alpha_k \in \R, i+1\le k\le n} \Big\{ 1 + \sum_{k=i+1}^n \alpha_k^2 - 2\alpha_n \left< e_i, v_n \right> - \sum_{k=i+1}^{n-1} 2\alpha_k \alpha_n \left< e_k, v_n\right> \Big\}.
    \end{aligned}
\end{equation*}
Using $|2\alpha_k \alpha_n \left< e_k, v_n\right> | \le \alpha_k^2 + \alpha_n^2|\left< e_k, v_n\right>|^2$ and $|2\alpha_n \left< e_i, v_n \right>| \le (1-c^2) + (1-c^2)^{-1}\alpha_n^2 |\left< e_i, v_n \right>|^2$, we then have
\begin{equation*}
\begin{aligned}
    |\tilde{v}_i|^2 & \ge \inf_{\alpha_n \in \R} \Big\{ 1 + \alpha_n^2 - (1-c^2) - (1-c^2)^{-1} \alpha_n^2 |\left< e_i, v_n \right>|^2 - \alpha_n^2 \sum_{k=i+1}^{n-1} |\left< e_k, v_n\right>|^2 \Big\} \\
    & \ge\inf_{\alpha_n \in \R} \Big\{ c^2 + \alpha_n^2 - \alpha_n^2 (1-c^2)^{-1} \sum_{k=i}^{n-1} |\left< e_k, v_n\right>|^2 \Big\} \\
    & \ge c^2,
    \end{aligned}
\end{equation*}
where we have used \eqref{eq.ei.vn} in the last inequality.

It remains to justify $(iii)$. In fact, the uniform non-degeneracy of $\tilde v_i$ \eqref{eq.tildvi-nd} guarantees that \eqref{defvi} is well-defined and thus, each $v_i$ is a $C^\infty$ function of $v_n$. Then the chain rule implies  that $|\nabla v_i| \lesssim |\nabla v_n|$. The lemma follows.
\ep

\begin{corollary} \label{corviCM}
Let $\Omega_0$ be a Lipschitz graph domain. Then the construction of $\{v_1,\dots,v_n\}$ given by Lemma \ref{lemdefvi} satisfies:
\[ \delta \nabla v_i \in \CMs(M),\]
where $M$ depends on the Lipschitz constant of $\partial \Omega_0$.
\end{corollary}

\bp
Simple consequence of Lemma \ref{lemdefvi} and Theorem \ref{ThGFE2}.
\ep

\subsection{Quantifying small Lipschitz perturbations of Lipschitz graph domains.}\label{S.v}

In the rest of the section, \[\Omega_0:=\{ (x,t) \in \R^{n-1}\times\R,\, t> g(x)\}\] is a Lipschitz graph domain, with a (possibly large) constant $M$. $G$ is the Green function with pole at $\infty$ for the Laplacian on $\om_0$.  

\medskip

The following notation will be used for the rest of the article. First, $v_n := \frac{(\nabla G)^T}{|\nabla G|}$ as in Lemma \ref{lemdefvi}. Then, we define
\begin{equation} \label{deftXyX}
t(X) := \frac{G(X)}{|\nabla G(X)|}, \qquad  y(X) := \pi\Big( X - \frac{G(X)(\nabla G(X))^T}{|\nabla G(X)|^2} \Big) = \pi\Big( X - t(X)v_n(X)\Big).
\end{equation}
where we recall that $\pi$ is the orthogonal projection on the first $(n-1)$-coordinates.
We emphasize that in our definition, $t:\om_0\to\R$ is a function and should not be confused with the variable $t\in\R$. We deliberately choose this notation, as we wish to view the function as playing the same role, in $\om_0$, of the special variable $t$ in $\Rn_+$. In fact, by Lemma \ref{lemNG<G/d}, \eqref{H1}, Theorem \ref{ThGFE2}, and  Corollary \ref{corviCM}, we have that 
\begin{equation} \label{t=d}
t\approx \delta := \dist(.,\partial \Omega_0);
\end{equation}
\begin{equation} \label{NvnisCM}
t \nabla v_i \in \CMs(CM) \quad \text{for }i\in\set{1,\dots,n},
\end{equation}
and
\begin{equation} \label{Nt-vnisCM}
\nabla t - v_n^T \in \CMs(CM).
\end{equation}

\begin{lemma}\label{lemDor}
Let $\varphi :\, \R^{n-1} \to \R$ be a Lipschitz function. Define $\beta_\varphi(x,r)$ as the $\beta$-number
\[\beta_\varphi(x,r) := \frac1r \inf_{a\in\Rn, B\in\R^{(n-1)\times n}} \fint_{B(x,r)} |\vp(z)-a-(z-x)B| \, dz \qquad \text{ for } x\in \R^{n-1} \text{ and } r >0.\]
Then 
\[X \mapsto \beta_\varphi (y(X),t(X)) \in CM_{sup}(C_M\|\nabla \varphi\|_\infty),\]
where $t(X)$ and $y(X)$ are as in \eqref{deftXyX}, and where $C>0$ is independent of $\varphi$ (depends only on the dimension $n$, and the Lipschitz constant $M$ of $g$).
\end{lemma}

\bp
We start from the Dorronsoro theorem (\cite{Dor85}), which states that $|\beta_\varphi(y,t)|^2 dy\, dt/t$ is a Carleson measure on $\R^{n-1}\times(0,\infty)$, with Carleson norm smaller than $C\|\nabla\varphi\|_{L^\infty}$, that is,
 \begin{equation}\label{eq.Dorr}
     \int_{B_{\R^{n-1}}(x,r)}\int_0^r|\beta_\varphi(z,s)|^2\frac{ds}{s}dz\le C_{Dor}\|\nabla\varphi\|_{\infty} r^{n-1} \qquad \text{ for $x\in \R^{n-1}$ and $r>0$.}
 \end{equation}
 We want to prove that 
\begin{equation}\label{Dor.claim}
     \int_{B(x,r) \cap \Omega_0} |\beta_\varphi(y(X),t(X))|^2\frac{dX}{\delta(X)}  \le C\|\nabla\varphi\|_{\infty} r^{n-1} \qquad \text{ for $x\in \partial \Omega_0$ and $r>0$.}
\end{equation}
The argument is straightforward and relies on Fubini's theorem. The details are given below. 

\medskip

Take $X \in \Omega_0$, $s\in [t(X),2t(X)]$, and 
$z\in B_{\R^{n-1}}(\pi(X),t(X))$. From the expression of $y(X)$ in \eqref{deftXyX}, one can see that 
\[|y(X) - \pi(X)| \leq \pi(t(X)v_n(X)) \leq t(X),\]
hence $|z-y(X)| \leq 2t(X) \leq 2s$, and then that
\begin{equation}\label{Dor.Ball_inc}
    B(y(X),t(X)) \subset B(z,3s).
\end{equation} 
Observe that the function $(y,t)\mapsto \beta_{\vp}(y,t)$ 
enjoys the following property: for $(y,t)$ and $(y',t')\in\R^{n-1}\times(0,\infty)$ satisfying the inclusion of balls $B(y',t')\subset B(y,t)$ and $t'\ge ct$ for some $c>0$, we have $\beta_{\vp}(y',t')\le C_{n,c}\beta_{\vp}(y,t)$.
So we deduce from \eqref{Dor.Ball_inc} that 
\[\beta_\varphi(y(X),t(X)) \lesssim \beta_\varphi(z,3s) \quad \text{ for $z\in B_{\R^{n-1}}(\pi(X),t(X))$ and $s\in [t(X),2t(X)]$},\] and so, for all $X\in \Omega_0$
\begin{multline*}
 \beta_\varphi((y(X),t(X)) \lesssim t(X)^{1-n} \int_{B(\pi(X),t(X))} \int_{t(X)}^{2t(X)} \beta_\varphi(z,3s) \, \frac{ds}{s} \, dz \\
 \lesssim \left( t(X)^{1-n} \int_{B(\pi(X),t(X))} \int_{3t(X)}^{6t(X)} |\beta_\varphi(z,s)|^2 \, \frac{ds}{s} \, dz \right)^\frac12.
\end{multline*}

\medskip

Take now $x\in \partial \Omega_0$ and $r>0$. By Fubini's theorem,
\begin{multline*}
\int_{B(x,r) \cap \Omega_0} |\beta_\varphi(y(X),t(X))|^2\frac{dX}{\delta(X)} \\\lesssim \int_{B(x,r) \cap \Omega_0} t(X)^{1-n} \left( \int_{B(\pi(X),t(X))} \int_{3t(X)}^{6t(X)} |\beta_\varphi(z,s)|^2 \, \frac{ds}{s} \, dz \right) \frac{dX}{\delta(X)}  \\
= \int_{\R^{n-1}} \int_0^\infty |\beta_\varphi(z,s)|^2 \frac{ds}{s} \, dz \Big( \underbrace{\int_{B(x,r) \cap \Omega_0} \frac{t(X)^{1-n}}{\delta(X)} \1_{B(\pi(X),t(X))}(z) \1_{[3t(X),6t(X)]}(s) \, dX}_{=:f(z,s)} \Big).
\end{multline*}
We claim that 
\begin{equation} \label{Dor.claim2}
f(z,s) \lesssim \1_{B(\pi(x),C_Mr)}(z) \1_{(0,C_M r)}(s).
\end{equation}
for some $C_M \geq 2$ independent of $\varphi$, $z$ and $s$. Once the claim \eqref{Dor.claim2} is established, we deduce that
\begin{multline*}\int_{B(x,r) \cap \Omega_0} |\beta_\varphi(y(X),t(X))|^2\frac{dX}{\delta(X)} \lesssim \int_{B(\pi(x),C_M r)} \int_0^{C_M r} |\beta_\varphi(z,s)|^2 \frac{ds}{s} \, dz \\ \leq C_{Dor} \|\nabla \varphi\|_{\infty} (C_M r)^{n-1} \lesssim \|\nabla \varphi\|_{\infty} r^{n-1},
\end{multline*}
meaning that \eqref{Dor.claim} holds as desired for the lemma.

\medskip

It remains to prove \eqref{Dor.claim2}. First, we observe that $f(z,s) \neq 0$ only if there exists $X\in B(x,r)\cap \Omega_0$ such that $|z-\pi(X)| \leq t(X)$ and $s \in [3t(X),6t(X)]$.
By \eqref{t=d}, we have then 
\[ s \leq 6t(X) \lesssim \delta(X) \leq r\]
and 
\[|z- \pi(x)| \leq |z-\pi(X)| + |\pi(X-x)| \leq t(X) + r \lesssim r,\]
meaning that there indeed exists $C_M\geq 2$ such that $f(z,s) \equiv 0$ outside of $B(\pi(x),C_Mr) \times (0,C_Mr)$. 

Note also that by \eqref{t=d} and the fact that $t(X) \approx s$, 
\[f(z,s) \lesssim s^{-n} |\{X\in \Omega_0, \, z\in B(\pi(X),t(X)), \, s\in [3t(X),6t(X)]\}|.\]
However, $z\in B(\pi(X),t(X))$ and $s\in [3t(X),6t(X)]$ implies that 
\begin{equation}\label{Dor.piX-z}
    |\pi(X) - z| \leq t(X) \leq \frac{s}{3}
\end{equation} 
and 
\[ |X - (\pi(X),g(z))| \leq |X - (\pi(X),g(\pi(X)))| + |g(z) - g(\pi(X))| \leq M\delta(X) + M |\pi(X) - z|\]
since $g$ is $M$-Lipschitz.
By invoking \eqref{t=d} again, the latter bound can be continued as follows
\[ |X - (\pi(X),g(z))| \le M\delta(X) + M \frac s3 \lesssim s.\]
This together with \eqref{Dor.piX-z} gives that 
\[
\abs{X-(z,g(z))}\le |\pi(X)-z|+|X - (\pi(X),g(z))|\lesssim s.
\]
It means that there exists $C$ such that
\[\{X\in \Omega_0, \, z\in B(\pi(X),t(X)), \, s\in [3t(X),6t(X)]\} \subset B((z,g(z)),Cs),\]
and thus
\[f(z,s) \lesssim s^{-n} |B((z,g(z)),Cs)| \lesssim 1.\]
The claim \eqref{Dor.claim2} and the lemma follows.
\ep

\medskip

\medskip

We set now $\Omega:= \Phi(\Omega_0)$, where $\Phi: \R^{n} \to \R^n$ is a bi-Lipschitz map satisfying $\|\nabla \Phi - I\|_{\infty}\leq \epsilon$, where $\epsilon>0$ will be small. We want to construct a ``better'' bi-Lipschitz change of variable $\rho$ from $\Omega_0$ to $\Omega$, which preserves the structure of $\Omega_0$. By preserving the structure of $\Omega_0$, we mean that at least the Jacobian $\nabla \rho$ is close to identity, in the sense that the errors are Carleson measures on $\om_0$ with small constant.

\medskip

We first need a mollifier. We let $\eta \in C^\infty_0(\R^{n-1})$ be a {\bf radial} function such that $\eta \geq 0$, $\int_{\R^{n-1}} \eta \, dx = 1$, and $\eta \equiv 0$ outside $B_{\R^{n-1}}(0,1)$. Then, for $t>0$, we define $\eta_t(x) := t^{1-n} \eta(x/t)$.

\medskip

We view $\om$ as a perturbation of $\om_0$, and the quantity that we want to control is $\lambda: \om_0\to\Rn$ defined by
\begin{equation} \label{deflambda}
\lambda(X) := \int_{\R^{n-1}}  h(z) \, \eta_{t(X)}\big(y(X) - z\big) \, dz \quad\text{for }X\in\om_0,
\end{equation}
where $h$ is the displacement
\begin{equation}\label{defh}
h(z):= \Phi(z,g(z)) - (z,g(z)).
\end{equation}
Note that $h:\R^{n-1}\to \Rn$ is Lipschitz, since
\begin{equation} \label{heLip}
    |\nabla h(z)| \leq |\nabla \Phi((z,g(z)) - I| \sqrt{1+|\nabla g(z)|^2} \leq C_M \epsilon.
\end{equation}

\medskip

The main result of this subsection is:

\begin{theorem} \label{Thlambda}
Let $\Omega_0:=\{ (x,t) \in \R^n,\, t> g(x)\}$, where $g:\R^{n-1}\to\R$ is a $M$-Lipschitz function, and let $\Phi:\, \R^n \to \R^n$ be a bi-Lipschitz change of variable satisfying $\|\nabla \Phi - I\|_\infty \leq \epsilon$. Then the quantity $\lambda$ defined in \eqref{deflambda}--\eqref{deftXyX} has the following properties:
\begin{enumerate}[(i)]
\item $\|\nabla \lambda\|_\infty \leq C_M \epsilon$.
\item $\partial_{v_n} \lambda \in \CMs(C_M\epsilon)$.
\item There exists a decomposition $\nabla \lambda  =: \mathbf{B} + \mathbf{C}$ verifying $\|\mathbf B\|_\infty \leq C_M\epsilon$, $v_n \mathbf B = \mathbf{0}_{\R^n}$, and $|\mathbf{C}| + t|\nabla \mathbf{B}| \in \CMs(C_M\epsilon)$.
\end{enumerate}
The constant $C_M>0$ in (i), (ii) and (iii) is independent of $\epsilon$, meaning that it depends only on $M$, the dimension $n$, the choice of $\eta$, and the constant in \eqref{H1}. 
\end{theorem}

\bp {\bf Step 0: Preliminaries.} 
In the proof, we write $t$ for $t(X)$, $y$ for $y(X)$, $v_n$ for $v_n(X)$, and $G$ for $G(X)$; this considerably lightens the notation. Now, \eqref{deflambda} becomes
\[\lambda(X):= \int_{\R^{n-1}} h(z) \eta_{t}\big(y - z\big) \, dz.\]
Building on estimates \eqref{NvnisCM}, \eqref{Nt-vnisCM} and the observation that 
\[
\nabla y=\pi(\nabla X)-\pi(\nabla(tv_n))=I_{n\times(n-1)}-\pi\br{(\nabla t-v_n^T)v_n}-\pi(v_n^Tv_n)-\pi(t\nabla v_n),
\]
where $I_{n\times (n-1)}$ is the $n\times (n-1)$ matrix that takes the value $1$ on the diagonal and 0 elsewhere, we also have
\begin{equation} \label{Ny-stuffisCM}
\big[ \nabla y -  I_{n\times (n-1)} + (v_n)^T \pi(v_n) \big]\in \CMs(CM).
\end{equation}
In particular, the estimates \eqref{NvnisCM}, \eqref{Nt-vnisCM} and \eqref{Ny-stuffisCM} imply that (see Remark~\ref{re.CMsup})
\begin{equation} \label{allinLinfty}
\|\nabla t\|_\infty + \|t\nabla v_n\|_\infty + \| \nabla y \|_\infty \leq C(1+M^{1/2}) .
\end{equation}

\medskip

\noindent {\bf Step 1: proof of (i).} 
Now, we can compute $\nabla \lambda$. Take $a\in \R^n$ to be defined later. We have, since $\int \eta_t = 1$,
\begin{multline} \label{defNlambda}
\nabla \lambda(X) =\nabla(\lambda - a) =  \nabla_X \left( \int_{\R^{n-1}} [h(z)-a]  \eta_{t}\big(y - z\big) \, dz \right) \\ 
= (\nabla t) \left( \int_{\R^{n-1}} [h(z)-a] \partial_t \eta_{t}\big(y - z\big) \, dz \right)  +  \left( \int_{\R^{n-1}} [h(z)-a] (\nabla y) \nabla_y \eta_{t}\big(y - z\big)  \, dz \right),
\end{multline}
where $(\nabla y)  \nabla_y \eta_{t}$ is seen as a matrix multiplication ($\nabla y$ has dimensions $n\times (n-1)$ and $\nabla_y\eta_t$ is a column vector of dimension $n-1$). Check that $\nabla_y \eta_t$ and $\partial_t \eta_t$ are both supported in $B_{\R^{n-1}}(0,t)$, and verify $\|\nabla_y \eta_t\|_\infty +  \|\partial_t \eta_t\|_\infty \lesssim  t^{-n}$. So we obtain that 
\[|\nabla \lambda(X)| \lesssim t^{-n} \left( \int_{B(y,t)} |h(z)-a| \, dz\right) (|\nabla t| + |\nabla y|) \lesssim \frac{M^{1/2}+1}t \fint_{B(y,t)} |h(z) - a|\, dz\]
by \eqref{allinLinfty}. We choose $a:= \fint_{B(y,t)} h\, dz$, and since $h$ is $\epsilon$-Lipschitz, we get that $|h(z) - a| \leq \epsilon t $ for all $z\in B(y,t)$. So 
\[|\nabla \lambda(X)| \leq C_M\epsilon.\]
with a constant independent of $X$. Conclusion (i) of the Theorem follows.

\medskip

\noindent {\bf Step 2: proof of (ii).} We need to be finer. We compute 
\begin{align*} 
v_n\nabla \lambda(X)
&= \left( \int_{\R^{n-1}} h(z) \partial_t \eta_{t}\big(y - z\big) \, dz \right) (\partial_{v_n} t) 
 +  \left( \int_{\R^{n-1}} h(z)  v_n (\nabla y)  \nabla_y \eta_{t}\big(y - z\big)   \, dz \right)\\
 &= : I_1 + I_2.
\end{align*}
Observe that $\nabla_y\eta_t(y-z)=-\nabla_z\eta_t(y-z)$, and so 
\[
\int_{\R^{n-1}}   v_n (\nabla y)  \nabla_y \eta_{t}\big(y - z\big)   \, dz=0 ,
\]
which implies that 
\[
I_2=\int_{\R^{n-1}} \br{h(z)-a'} \big[ v_n (\nabla y)  \nabla_y \eta_{t}\big(y - z\big) \big]  \, dz \quad\text{for any }a'\in\R^n.
\]
Note that \eqref{Ny-stuffisCM} gives that
\[v_n (\nabla y) = v_n \big( \nabla y - I_{n\times (n-1)} + (v_n)^T \pi(v_n) \big) \in \CMs(CM).\]
Hence
\[|I_2| \leq \underbrace{ \left( \int_{B(y,t)} |h(z)-a'| |\nabla_y \eta_t(y-z)|\, dz \right)}_{\leq C\epsilon \text{ as in Step 1}} | v_n(\nabla y)| \in \CMs(CM\epsilon).\]
As for $I_1$, we use \eqref{allinLinfty} to get that
\[
|I_1| \lesssim (1+M^{1/2}) \Big| \underbrace{\int_{\R^{n-1}} h(z)\, \partial_t \eta_{t}\big(y - z\big) \, dz}_{I_{11}} \Big|
\]
To deal with $I_{11}$, we observe that 
\[ \partial_t \eta_t(z) = \diver_z  [\tilde \eta_t] (z),\]
with $\tilde \eta(z) = - z\eta(z)$ and $\tilde \eta_t(z) = t^{1-n} \tilde \eta(z/t)$. Denote by $(z_1,\dots,z_{n-1})$ the coordinates of $z$. By the divergence theorem, the identity implies that for any $a\in\R$,
\[
\int_{\R^{n-1}} (a+z_i) \partial_t \eta_{t}\big(y - z\big) \, dz = -\int_{\R^{n-1}} \nabla_z z_i \cdot \tilde \eta_t(y-z) \, dz = 0 \]
since $\tilde \eta$ is odd (because $\eta$ is radial). We deduce that for any constant vector $a\in\R^n$ and for any constant matrix $B\in\R^{(n-1)\times n}$, 
we have 
\[ \int_{\R^{n-1}} [a+ (z-y) B] \, \partial_t \eta_{t}\big(y - z\big) \, dz= 0,\]
and hence 
\[ 
I_{11} = \int_{\R^{n-1}} [h(z)-a - (z-y)B] \, \partial_t \eta_{t}\big(y - z\big) \, dz.\]
Consequently, 
\begin{equation}\label{eq.Ilebeta}
    |I_{11}| = |I_{11}(X)|\lesssim \beta_h\br{y(X),t(X)} \in CM_{sup}(C_M\epsilon),
\end{equation}
by Lemma \ref{lemDor} and \eqref{heLip}. Step 2 follows.

\medskip

\noindent {\bf Step 3: proof of (iii).} The computations are similar to Step 2, although more involved. Morally, we will proceed as in Step 2, but for all the coordinates instead just for $v_n$. The terms that we can control by Carleson measures at the `first derivative' level will be put in $\mathbf C$. The others will be placed in $\mathbf B$ and we shall differentiate them an additional time to get the control in terms of Carleson measure. We have 
\begin{align*} 
\nabla \lambda(X) 
 =&  (\nabla t) \left( \int_{\R^{n-1}} h(z) \partial_t \eta_{t}\big(y - z\big) \, dz \right)   +  \left( \int_{\R^{n-1}} (\nabla y) \nabla_y \eta_{t}\big(y - z\big) h(z)\, dz \right) \\
 = & (\nabla t) \left( \int_{\R^{n-1}} h(z) \partial_t \eta_{t}\big(y - z\big) \, dz \right) \\
& +  \left( \int_{\R^{n-1}} \big(\nabla y - I_{n\times (n-1)} + (v_n)^T \pi(v_n)\big) \nabla_y \eta_{t}\big(y - z\big) h(z) \, dz \right) \\
& + \left( \int_{\R^{n-1}} \big(I_{n\times (n-1)} - (v_n)^T \pi(v_n)\big) \nabla_y \eta_{t}\big(y - z\big) h(z) \, dz \right)
 = : J_1 + J_2 + {\bf B},
\end{align*}
where 
\[
{\bf B}:= \left( \int_{\R^{n-1}} \big(I_{n\times (n-1)} - (v_n)^T \pi(v_n)\big) \nabla_y \eta_{t}\big(y - z\big) h(z) \, dz \right).
\]
We write $\mathbf C := J_1 + J_2$. The control on $\mathbf C$ is like in Step 2: specifically, $J_1$ can be treated like $I_{1}$,  $J_2$ is treated like $I_2$ (using, instead of $v_n(\nabla y) \in \CMs(CM)$ like for $I_2$, the Carleson estimate \eqref{Ny-stuffisCM}).

For the $L^\infty$-bound on $\mathbf B$, we observe that 
\[
0=\int_{\R^{n-1}}  \big(I_{n\times (n-1)} - (v_n)^T \pi(v_n)\big) \nabla_y \eta_{t}\big(y - z\big) \, dz,
\]
and so 
\[
\mathbf{B}=  \int_{\R^{n-1}}  \big(I_{n\times (n-1)} - (v_n)^T \pi(v_n)\big) \nabla_y \eta_{t}\big(y - z\big) [h(z)-a'] \, dz\quad\text{for any }a'\in\R^n.
\]
Then the estimate $\|\mathbf B\|_\infty \leq C_M\epsilon$ is a simpler variant of the $L^\infty$-bound of $\nabla \lambda$ treated in Step 1. Moreover, observe that 
\[v_n  \big(I_{n\times (n-1)} - (v_n)^T \pi(v_n)\big) =\pi(v_n)-v_n(v_n)^T\pi(v_n) = 0,\]
which entails that $v_n \mathbf B = 0$. 

\medskip

We want now to derive a Carleson measure estimate for $\nabla \mathbf B$. Let $\{\mathbf b_j\}_{1\leq j\leq n}$ be the $n$ rows of $\mathbf B$, and we want thus to study the quantities $\nabla \mathbf b_j$. Write $v_n=:\big((v_n)_1,\dots,(v_n)_n\big)$ and let $\{e_1,\dots,e_n\}$ be the Cartesian orthonormal basis. Then
\[ 
\mathbf b_j(X) =  \int_{\R^{n-1}} h(z) \Big< \pi(e_j - (v_n)_j v_n) ,    \nabla_y \eta_{t}(y - z)  \Big> \, dz .
\]
Differentiating in $X_i$, we have by the chain rule that 
\begin{multline*} 
|\partial_i \mathbf b_j(X)| \leq  \left| \int_{\R^{n-1}}  \nabla_y  \partial_t \eta_{t}\big(y - z\big) h(z) \, dz \right|  |\pi(e_j - (v_n)_j v_n)|  |\partial_i t| \\
+  \left| \int_{\R^{n-1}}  \nabla_y^2 \eta_{t}\big(y - z\big)  h(z) \, dz \right| |\pi(e_j - (v_n)_j v_n)|  |\nabla y|  \\
+ \left| \int_{\R^{n-1}}    \nabla_y \eta_{t}(y - z) h(z) \, dz\right| |\partial_i[(v_n)_jv_n]|
=: K_1 + K_2 + K_3.
\end{multline*}
The term $K_3$ is simple: we have that 
\[
K_3=\left| \int_{\R^{n-1}}    \nabla_y \eta_{t}(y - z) [h(z)-a] \, dz\right| |\partial_i[(v_n)_jv_n]|\quad \text{for any }a\in\R^n,
\]
and we can prove that
\[ \left| \int_{\R^{n-1}} [h(z)-a]   \nabla_y \eta_{t}(y - z) \, dz\right| \lesssim \epsilon,\]
with the same strategy as in Step 1, and \eqref{NvnisCM} gives that $ t\,\partial_i[(v_n)_jv_n] \in \CMs(CM)$. Hence $t|K_3| \in \CMs(CM\epsilon)$ as desired. For the two other terms $K_1$ and $K_2$, observe that $\|\pi(e_j - (v_n)_j v_n)\|_\infty \leq 2$ and
\[  \|\partial_i t\|_\infty  + \|\nabla y\|_\infty \lesssim (1+M)\]
by \eqref{allinLinfty}. So proving that $t\br{K_1+K_2} \in \CMs(C_M\epsilon)$ reduces to proving that $t|K_{11}| + t|K_{21}| \in \CMs(C_M\epsilon)$, where
\[ K_{11}:=  \int_{\R^{n-1}}  \nabla_y  \partial_t \eta_{t}\big(y - z\big) h(z) \, dz\]
and
\[ K_{12}:=  \int_{\R^{n-1}} \nabla_y^2 \eta_{t}\big(y - z\big) h(z) \, dz.\]
The strategy is the same as the one used for $I_{11}$. Observe that 
\[\nabla_y\eta_t(y-z)=-\nabla_z\eta_t(y-z) \text{ and } \int_{\R^{n-1}} \partial_t \eta_t \, dz = \int_{\R^{n-1}} \nabla_y \eta_t \, dz = 0.\]
Hence, by the divergence theorem, for any $a\in\R^n$ and any constant matrix $B\in\R^{(n-1)\times n}$,
\[\int_{\R^{n-1}} [a+ (z-x)B] \nabla_y \partial_t \eta_t(y-z) \, dz = \int_{\R^{n-1}} [a+(z-x) B] \nabla^2_y \eta_t(y-z) \, dz = 0,\]
and thus
\begin{multline*}
 |K_{11}| \leq \inf_{a\in\Rn, B\in\R^{(n-1)\times n}}  \left| \int_{\R^{n-1}} \nabla_y  \partial_t \eta_{t}\big(y - z\big) \big[h(z)-a - (z-x)B\big]  \, dz \right| \\
 \lesssim  \inf_{a\in\Rn, B\in\R^{(n-1)\times n}}  \frac1{t^2}  \fint_{B(y,t)} |h(z) -a - (z-x)B|\, dz=\frac{\beta_h(y,t)}{t}
 \end{multline*}
and similarly
\begin{multline*}
|K_{12}| \leq  \inf_{a\in\Rn, B\in\R^{(n-1)\times n}}  \left|  \int_{\R^{n-1}}  \nabla_y^2 \eta_{t}\big(y - z\big) \big[h(z)-a- (z-x)B)\big]\, dz \right| \\ \lesssim \inf_{a\in\Rn, B\in\R^{(n-1)\times n}}  \frac1{t^2} \fint_{B(y,t)} |h(z) - a - (z-x)B|\, dz=\frac{\beta_h(y,t)}{t}.
 \end{multline*}
We invoke Lemma \ref{lemDor} and \eqref{heLip}, as in Step 2, to obtain that $t|K_{11}| + t|K_{21}| \in \CMs(C_M\epsilon)$. The theorem follows.
\ep

\subsection{Construction of the change of variables}\label{S.cov}

Let $\{e_1,\dots,e_n\}$ be the Cartesian basis, $\{v_i,\dots,v_n\}$ be the orthonormal basis constructed in Lemma \ref{lemdefvi}, and $\lambda$ be the perturbation defined in \eqref{deflambda}--\eqref{deftXyX}.
We define then the vector $\bar w_i$ as
\begin{equation} \label{defwtwi}
\bar w_i(X) := v_i(X) [I+ \mathbf B(X)] \quad\text{for }X\in\om_0,
\end{equation}
where $\mathbf B$ is the matrix-valued function from Theorem \ref{Thlambda}, $(iii)$.

\begin{lemma} \label{lemdefwi}
Let $\Omega_0:=\{ (x,t) \in \R^n,\, t> g(x)\}$, where $g:\R^{n-1}\to\R$ is a $M$-Lipschitz function, and let $\Phi:\R^{n}\to\R^n$ be a bi-Lipschitz change of variables satisfying $\|\nabla \Phi - I\|_\infty \leq \epsilon$. There exists $\epsilon_2>0$ such that if $\epsilon\in (0,\epsilon_2]$, then there exists an orthonormal basis $\{w_1,\dots,w_n\}$ satisfying
\begin{enumerate}[(i)]
\item For each $i\in \{1,\dots,n\}$, $\Span\{w_1,\dots,w_i\} = \Span\{\bar w_1,\dots,\bar w_i\}$,
\item There exists $C>0$ such that, for any $i,j\in \{1,\dots,n\}$, we have\footnote{Recall that $I_{ij}$ are the coefficients of the $n\times n$ identity matrix.} $|\left<w_i,v_j \right> - I_{ij}|\leq C\epsilon$,
\item There exists $C>0$ such that, for any $i,j\in \{1,\dots,n\}$, we have
\[t\abs{\nabla  \langle w_i,v_j\rangle } \in \CMs(C\epsilon),\]
where $t$ is as defined in \eqref{deftXyX}.
\end{enumerate}
In the lemma, $\epsilon_2$ and $C$ depend on $M$ (on top of the the dimension $n$ and the choice of $\eta$).
\end{lemma}

\bp
The lemma is obtained again with the Gram-Schmidt process. We define by induction 
\begin{equation} \label{defwi}
\tilde w_1=\bar w_1, \quad
\tilde w_i:=  \bar w_i - \sum_{1\le k<i} \left< \bar w_i, w_k\right> w_k \quad\text{for }2\le i\le n, \quad \text{ and } \quad  w_i := \dfrac{\tilde w_i}{|\tilde w_i|}.
\end{equation}

We prove by induction on $i\in \{1,\dots,n\}$ that 
\begin{enumerate}
\item  $\left<\tilde w_i,v_i \right>\geq 1-C\epsilon$, 
\item  $\big| |\left<w_i,v_i \right>| - 1\big| \leq C\epsilon$,
\item for all $j\in \{1,\dots,n\} \setminus \{i\}$,  $|\left<w_i,v_j \right>|\leq C\epsilon$, and
\item for all $j\in \{1,\dots,n\}$, we have $t\nabla  \left<w_i,v_j\right>  \in \CMs(C\epsilon)$.
\end{enumerate}
Note that (1) implies that $|\tilde w_i|\ge 1-C\epsilon$, which entails that the construction \eqref{defwi} is well defined. It is obvious from the construction that $(i)$ holds and that $\set{w_1,\dots,w_n}$ is an orthonormal basis.
\smallskip

\noindent {\em Initial Step:} We have $\tilde w_1 = \bar w_1$, and so 
\[\left<\tilde w_1,v_1 \right> = 1 + \left<v_1\mathbf B,v_1 \right> \geq 1 - C \epsilon\] 
by the $L^\infty$ bound of $\mathbf{B}$ in (iii) of Theorem \ref{Thlambda}, which yields conclusion (1). We also have
\[|\tilde w_1| \leq  |v_1| + |\mathbf B| \leq 1 + C \epsilon\] 
Altogether, $\big| |\tilde w_1| - 1 \big| \leq C\epsilon$, and then
\[\big| |\left<w_1,v_1 \right>| - 1\big| \leq \left| \frac{|\left<\tilde w_1,v_1 \right>| - 1}{|\tilde w_1|}\right| + \left| \frac{|\tilde w_1| - 1}{|\tilde w_1|}\right| \lesssim \epsilon\]
as long as $\epsilon_2$ is chosen small enough such that $|\tilde w_1|\geq 1-C\epsilon_2 \geq \frac12$. Since $w_1$ is a unit vector and $\{v_1,\dots,v_n\}$ is an orthonormal base, the fact that $\big| |\left<w_1,v_1 \right>| - 1\big|\leq C\epsilon$ immediately implies that $\left<w_1,v_j \right>| \leq C\epsilon$ when $j\neq 1$. So we have proved (2) and (3) when $i=1$. Finally
\[ |\nabla  \left<\tilde w_1,v_j\right>| = |\nabla \big[ \left<v_1 \mathbf B,v_j\right> \big] |\lesssim |\nabla \mathbf B| + |\mathbf B| (|\nabla v_1| + |\nabla v_j|),\]
meaning that $t \nabla  \left<\tilde w_1,v_j\right> \in \CMs(C\epsilon)$ by Theorem \ref{Thlambda} and Corollary \ref{corviCM}. Therefore, 
\[t\,\nabla \left[|\tilde w_1|^2\right] = t\,\nabla \left[ \sum_{i=1}^n  |\left<\tilde w_1,v_j\right>|^2\right] \in \CMs(C\epsilon)\]
and, since $|\tilde w_1|\geq \frac12$ by our choice of $\epsilon_2$,
\[t\,\nabla\big[\left<w_1,v_j\right> \big] = t\,\nabla \left[ \frac{\left<\tilde w_1,v_j\right>}{\sqrt{|\tilde w_1|^2}} \right] \in \CMs(C\epsilon).\]
The case $i=1$ follows.

\noindent {\em Induction step:} Let $i\in \{2,\dots,n\}$ and assume that (1)--(4) are true for all $1\leq k<i$. Then, using the expression of $\bar w_i$ and $\tilde w_i$  given in \eqref{defwtwi}--\eqref{defwi},
\begin{multline*}
 \left< \tilde w_i,v_i\right> =  1 +\left< v_i \mathbf B,v_i\right> - \sum_{k<i} |\left< v_i, w_k\right>|^2 -  \sum_{k<i} \left< v_i \mathbf B, w_k\right> \left<w_k, v_i\right> \\
\geq 1 - i\, |\mathbf B| - \sum_{k<i} |\left<w_k,v_i\right>|^2 \geq 1- C\epsilon
 \end{multline*}
  by (iii) of Theorem \ref{Thlambda} and (3) of the induction hypothesis. A similar argument yields that $|\tilde w_i| \leq 1 + C\epsilon$, which means, with the same reasoning as in the initiation step, that conclusions (2) and (3) hold too once we take $\epsilon_2$ small enough. To prove (4), we observe that 
\begin{multline*}
\left<\tilde w_i,v_j\right> = I_{ij} + \left<v_i\mathbf B,v_j\right> - \sum_{k<i} \left< v_i, w_k\right> \left< w_k, v_j\right>  -  \sum_{k<i} \left< v_i\mathbf B,w_k\right> \left<w_k, v_j\right> \\
= I_{ij} + \left< v_i \mathbf B,v_j\right> - \sum_{k<i} \left< v_i, w_k\right> \left< w_k, v_j\right>  -  \sum_{k<i} \sum_{\ell=1}^n  \left< v_i \mathbf B,v_\ell\right> \left<v_\ell,w_k\right> \left<w_k, v_j\right>,
 \end{multline*}
 where we have written $w_k=\sum_{\ell=1}^n\langle w_k,v_\ell\rangle v_\ell$ in the last equality. 
Direct computation shows that 
\[
|\nabla  \left<\tilde w_i,v_j\right>| \lesssim |\nabla\mathbf B| + |\mathbf B| \left(\sum_{\ell = 1}^n |\nabla v_\ell| \right) + (1+ |\mathbf B|) \sum_{k< i } \sum_{\ell=1}^n |\nabla \left<v_\ell,w_k\right>|,
\]
and thus $t \nabla  \left<\tilde w_i,v_j\right> \in \CMs(C\epsilon)$ by Theorem \ref{Thlambda} and (4) of the induction hypothesis. With the same argument as in the initiation step, we deduce that 
\[t \nabla  \left<\tilde w_i,v_j\right> \in  \CMs(C\epsilon)\]
The induction step and the lemma follows.
\ep

In addition to properties $(i)-(iii)$ in Lemma \ref{lemdefwi}, we shall also need the following estimate for $\nabla w_i$: for $i\in\set{1,\dots,n}$,
\begin{equation}\label{eq.gradwi}
    t\abs{\nabla w_i}\in \CMs(CM),
\end{equation}
which can be deduced by expressing $w_i=\sum_{j=1}^n\langle w_i,v_j\rangle v_j$ and then using \eqref{NvnisCM} and $(iii)$ of Lemma \ref{lemdefwi}.
\ms

Let $\mathcal O=\mathcal O(X)$ be the rotation matrix that sends the orthonormal basis $\{v_1,\dots,v_n\}$ constructed in Lemma \ref{lemdefvi} to the orthonormal basis  $\{w_1,\dots,w_n\}$ constructed in Lemma \ref{lemdefwi}. Written in matrix multiplication, it means that 
\begin{equation}\label{defO}
    w_i=v_i\mathcal O \text{ for all } i\in\set{1,\dots,n}, \quad \text{or }W=V\mathcal O,
\end{equation}
 where $V$ is the $n\times n$ (orthogonal) matrix with row vectors $v_1$, $\dots$, $v_n$, and $W$ is the $n\times n$ (orthogonal) matrix with row vectors $w_1$,$\dots$,$w_n$.

\begin{lemma} \label{lemO}
Let $\Omega_0:=\{ (x,t) \in \R^n,\, t> g(x)\}$, where $g:\R^{n-1}\to\R$ is a $M$-Lipschitz function. Let $\epsilon_2$ be the value given in Lemma \ref{lemdefwi}, and let $\Phi:\R^{n}\to\R^n$ be a bi-Lipschitz change of variables satisfying $\|\nabla \Phi - I\|_\infty \leq \epsilon_2$. Then the matrix $\mathcal O$ defined as in \eqref{defO}
verifies the following conditions:
\begin{enumerate}[(i)]
\item $|\mathcal O - I_{n\times n}| \le C\epsilon$,
\item $t\nabla \mathcal O \in \CMs(C\epsilon)$, where $t$ is defined as in \eqref{deftXyX}.
\end{enumerate}
In both cases, the constant $C$ depends on $M$, but not on $\epsilon$.
\end{lemma}

\bp
This lemma is actually a simple corollary of Lemma \ref{lemdefwi}. 
Let $\mathcal O_v$ be the transition matrix that express $\set{w_1,\dots,w_n}$ in terms of $\set{v_1,\dots,v_n}$, that is, $\mathcal O_v V=W$, and in particular, its $ij$-th entry is $\langle w_i,v_j\rangle$.  Lemma \ref{lemdefwi} then gives that $|\mathcal O_v - I_{n\times n}| \leq C\epsilon$ and $t \nabla \mathcal O_v \in \CMs(C\epsilon)$. Since $\mathcal O=V^{-1}\mathcal O_vV$,
we have that $|\mathcal O - I_{n\times n}| = |\mathcal O_v - I_{n\times n}| \leq C\epsilon$. By Corollary \ref{corviCM}, $t\abs{\nabla V}\in \CMs(C_M)$. So we get that   
\[t|\nabla \mathcal O| = t|\nabla (\mathcal O - I_{n\times n})| 
=t\abs{\nabla\left[V^T(\mathcal O_v-I_{n\times n})V\right]}
\leq 2 t|\nabla V| |\mathcal O_v - I_{n\times n}| + t|\nabla O_v| \in \CMs(C\epsilon).\]
The lemma follows.
\ep

We set the matrix function $A_{||}:\, \Omega_0 \to \R^{(n-1)^2}$ and the function $a: \Omega_0 \to \R$ to be such that 
\begin{equation} \label{defAX}
A_{||}(X) := \begin{bmatrix} \left< \bar w_1, w_1\right> & \dots & \left< \bar w_1, w_{n-1}\right> \\ \vdots & \ddots & \vdots \\  \left< \bar w_{n-1}, w_1\right> & \dots & \left< \bar w_{n-1}, w_{n-1}\right>  \end{bmatrix} = ( \left< \bar w_i, w_j\right>)_{1\leq i,j\leq n-1}
\end{equation} 
and
\begin{equation} \label{defaX}
a(X) := \det(A_{||}(X)).
\end{equation} 

\begin{lemma} \label{lemJ0}
Let $\Omega_0:=\{ (x,t) \in \R^n,\, t> g(x)\}$, where $g:\R^{n-1}\to\R$ is a $M$-Lipschitz function. Let $\epsilon_2$ be the value given in Lemma \ref{lemdefwi}, and let $\Phi:\R^{n}\to\R^n$ be a bi-Lipschitz change of variables satisfying $\|\nabla \Phi - I\|_\infty \leq \epsilon_2$. Then the matrix $J_0$ defined as
\begin{equation} \label{blockJ0}
J_0(X) := \begin{bmatrix} A_{||}(X) & \mathbf{0}^T \\ \mathbf{0} & a(X) \end{bmatrix}, \quad \text{where } \mathbf{0} = 0_{\R^{n-1}} = (0,\dots,0),
\quad\text{for }X\in\om_0
\end{equation}
verifies $\|J_0 - I\|_\infty \leq C\epsilon$ and  $t |\nabla J_0| \in \CMs(C\epsilon)$, where $t$ is defined as in \eqref{deftXyX}. The constant $C$ depends on $M$ but not on $\epsilon$.
\end{lemma}

\bp
Recall that 
\[\bar w_i := v_i (I+\mathbf B). \]
So  
\[ \|\left< \bar w_i, w_j\right> - I_{ij}\|_\infty \leq   \|\left< v_i, w_j\right> - I_{ij}\|_\infty + \|\mathbf B\|_\infty \lesssim \epsilon\]
by $(ii)$ of Lemma \ref{lemdefwi} and $(iii)$ of Theorem \ref{Thlambda}. It proves that $\|A_{||} - I_{(n-1)\times(n-1)}\|_\infty \lesssim \epsilon$.
Moreover, we have 
\begin{equation}\label{eq.wiwj}
    t\Big|\nabla\big[\left< \bar w_i, w_j\right>\big]\Big| \lesssim t \Big|\nabla\big[\left< v_i, w_j\right>\big]\Big| + t\|\mathbf B\|_\infty \br{|\nabla v_i|+|\nabla w_j|} + t|\nabla \mathbf B| \in \CMs(C\epsilon)
\end{equation}
by Lemma \ref{lemdefwi}, Theorem \ref{Thlambda}, Corollary \ref{corviCM}, and \eqref{eq.gradwi}. It proves that $t|\nabla A_{||}| \in \CMs(C\epsilon)$.

\medskip

Since $a$ is the determinant of $A_{||}=(a_{ij})_{1\le i,j\le n-1}$, it can be written as
\begin{align*}
 a - 1 & = \left( \prod_{i=1}^{n-1} a_{ii}\right) - 1 + \sum_{\sigma \in \mathfrak S_{n-1} \atop \sigma \neq id} \sgn(\sigma)  \prod_{i=1}^{n-1} a_{i,\sigma(i)} \\
 & = \sum_{k=1}^{n-1} (a_{kk}-1) \left( \prod_{i=k+1}^{n-1} a_{ii}\right) + \sum_{\sigma \in \mathfrak S_{n-1} \atop \sigma \neq id} \sgn(\sigma)  \prod_{i=1}^{n-1} a_{i,\sigma(i)},
  \end{align*}
  where we have used the convention that $\prod_{i=n}^{n-1}a_{ii}=1$.
 Observe  that each product above contains at least one element of $A_{||}-I$. Combined to the fact that the entries of $A_{||}$ are uniformly bounded (by $1+C\epsilon$), we get
  \[|a-1| \lesssim \|A_{||}-I\|_\infty \lesssim \epsilon.\]
Since $a$ is a linear combination of product of elements of $A_{||}$, we also easily have
\[t|\nabla a| \lesssim \sum_{i=1}^{n-1} \sum_{j=1}^{n-1} t\Big|\nabla\big[\left< \bar w_i, w_j\right>\big]\Big| \in \CMs(C\epsilon)\]
by \eqref{eq.wiwj}.
The lemma follows.
\ep

Given $\Omega_0:=\{ (x,t) \in \R^n,\, t> g(x)\}$, where $g:\R^{n-1}\to\R$ is a $M$-Lipschitz function. Given a bi-Lipschitz map $\Phi$ satisfying $\| \nabla \Phi - I\|_\infty \leq \epsilon$, we now define the change of variables $\rho$ from $\Omega_0$ to $\R^n$ as
\begin{equation} \label{defrho}
\rho(X) = X- t(X)v_n(X) + \lambda(X) + a(X) t(X) v_n(X) \mathcal O(X),
\end{equation}
where $\lambda(X)$ and $t(X)$ are defined in \eqref{deflambda}--\eqref{deftXyX}, $a$ defined in \eqref{defaX}, and $\mathcal{O}$ defined in \eqref{defO}. With a lighter notation, it becomes 
\begin{align} \label{defrho2}
\nonumber \rho(X) & = X- tv_n + \lambda + a t v_n\mathcal O \\ & = X + \lambda + a tv_n(\mathcal O-I) + (a-1)tv_n .
\end{align}
Let $\om= \Phi(\om_0)$. 
Observe that for $X\in\pom_0$, $t(X)=0$, $y(X)=\pi(X)$, and $\lambda(X)=h(\pi(X))$, where $h$ is defined as in \eqref{defh}. So  $\rho(X)=X+h(\pi(X))=\Phi(X)$ for $X\in\pom_0$, which implies that $\rho(\pom_0)=\pom$. 
\smallskip

We want to prove that $\rho$ is a bijection between $\Omega_0$ and $\Omega$. 
Moreover, we want to derive estimates for the Jacobian matrix of the map $\rho$ in the bases $\{v_1,\dots,v_n\}$ and $\{w_1,\dots,w_n\}$, which we define as the matrix $J$. That is, the $ij$-th entry of the Jacobian matrix $J$ is 
\begin{equation}\label{defJ}
    J_{ij}=\left< \partial_{v_i} \rho,w_j\right>.
\end{equation}

\begin{theorem} \label{ThrhoJ}
Let $\Omega_0:=\{ (x,t) \in \R^n,\, t> g(x)\}$, where $g$ is a $M$-Lipschitz function. There exists $\epsilon_2=\epsilon_2(M) \in (0,1]$ such that for any $\epsilon \in [0,\epsilon_2]$ and for any bi-Lipschitz change of variables $\Phi:\R^{n}\to\R^n$  satisfying $\|\nabla \Phi - I\|_\infty \leq \epsilon$, the Jacobian matrix $J$ defined in \eqref{defJ} for the map $\rho$ defined in \eqref{defrho} satisfies
\begin{enumerate}[(i)]
\item $\|J- I_{n\times n}\|_\infty \leq C\epsilon$,
\item $J_1:= J-J_0 \in \CMs(C\epsilon)$, where the Carleson condition is with respect to domain $\om_0$.
\end{enumerate}
In both cases, the constant $C>0$ depends on $M$, but not on $\epsilon$.
\end{theorem}

\bp
Direct computation shows that 
\[
\partial_{v_i} \rho = v_i + (\partial_{v_i} \lambda) + atv_n\partial_{v_i}\mathcal O +  (\partial_{v_i} t)(av_n\mathcal O - v_n) +  t \big[\partial_{v_i} v_n \big] (a \mathcal O - I)  + (\partial_{v_i} a) tv_n\mathcal O. 
\]
We have $ v_i + (\partial_{v_i} \lambda)= \bar w_i + v_i\mathbf C$, where we recall that $\nabla \lambda = \mathbf B + \mathbf C$ is the decomposition given in Theorem \ref{Thlambda} and $\bar w_i$ is given in \eqref{defwtwi}. Moreover, $v_n\mathcal O = w_n$ by the definition on $\mathcal{O}$ in \eqref{defO}. So we can write 
\begin{multline}\label{eq.dvirho}
\partial_{v_i} \rho = \bar  w_i + v_i \mathbf  C  + atv_n\partial_{v_i} \mathcal O + \left< \nabla t - v_n^T,v_i\right> (a w_n - v_n) +\left<v_n,v_i\right> (a w_n - v_n) \\
+ t  \big[\partial_{v_i} v_n\big] (a \mathcal O - I)+  (\partial_{v_i} a) tw_n.
\end{multline}
We claim that the entries of the matrix $J_0$ (defined in \eqref{blockJ0}) are 
\begin{equation}\label{eq.J0'}
    (\widetilde{J_0})_{ij} := \Big<  \bar  w_i + \left<v_n,v_i\right> (a w_n - v_n),  w_j\Big>.
\end{equation} 
Indeed,
\begin{itemize}
\item When $i\in \{1,\dots,n-1\}$ and $j\in \{1,\dots,n\}$, we have $\left<v_n,v_i\right> = 0$ and thus
\[(\widetilde{J_0})_{ij}= \left<  \bar  w_i,  w_j\right>.\]
So $(\widetilde{J_0})_{ij} = (A_{||})_{ij} = (J_0)_{ij}$ when $i<n$ and $j<n$. Furthermore, recall that by construction (see (i) of Lemma \ref{lemdefwi}), we have $\left< \bar w_k, w_\ell \right> = 0$ whenever $\ell >k$, which means particularly that $\langle \bar w_i, w_n\rangle =0$ when $i<n$. Hence, $(\widetilde{J_0})_{in} = 0 = (J_0)_{in}$.
\item When $i=n$ and $j\in \{1,\dots,n\}$, we have $\bar w_i = v_i$ and $\left<v_n,v_i\right>=1$, 
so
\[ (\widetilde{J_0})_{nj}:= \left< a w_n,w_j \right> = a I_{nj}.\]
So the last line of $\widetilde{J_0}$ equals that of $J_0$.
\end{itemize}
By \eqref{eq.dvirho} and \eqref{eq.J0'}, one sees that the matrix $J_1$ has entries
\[
(J_1)_{ij} := \Big< 
  v_i \mathbf  C + atv_n\partial_{v_i} \mathcal O + \left< \nabla t - v_n^T,v_i\right> (a w_n - v_n) 
+ t  \big[\partial_{v_i} v_n\big] (a \mathcal O - I) +  (\partial_{v_i} a) tw_n
,\,w_j\Big>.
\]
To prove that $J_1 \in \CMs(C\epsilon)$, observe that
\begin{multline*}
    |(J_1)_{ij}| \lesssim 
\underbrace{|\mathbf C|}_{\in \CMs(C\epsilon)} 
+ \underbrace{t|\nabla \mathcal O|}_{\in \CMs(C\epsilon)} 
+ \underbrace{|\nabla t - v_n^T|}_{\in \CMs(C)} \underbrace{\br{|a-1|+|w_n - v_n|}}_{\leq C\epsilon} \\
+ \underbrace{\br{|a-1|+|\mathcal O- I|}}_{\leq C\epsilon} \underbrace{t|\nabla v_n|}_{\in \CMs(C)}
+\underbrace{t|\nabla a|}_{\in \CMs(C\epsilon)}
\in \CMs(C\epsilon),
\end{multline*} 
by Theorem \ref{Thlambda}, \eqref{NvnisCM}, \eqref{Nt-vnisCM}, Lemma \ref{lemdefwi}, Lemma \ref{lemO}, and Lemma \ref{lemJ0}. 

We have proved that $J$ satisfies the conclusion $(ii)$ of the Theorem. To prove $(i)$, observe that $J_1\in \CMs(C\epsilon)$ implies in particular $\|J_1\|_\infty \lesssim \epsilon$. In addition, $\|J_0 - I\|_\infty \lesssim \epsilon$ by Lemma \ref{lemJ0}. So we conclude that
\[\|J-I\|_\infty \leq \|J_0 - I\|_\infty + \|J_1\|_\infty \lesssim \epsilon,\]
as desired.
\ep

\begin{theorem} \label{Thrhodiffeo}
Let $\Omega_0:=\{ (x,t) \in \R^n,\, t> g(x)\}$, where $g$ is a $M$-Lipschitz function. There exists $\epsilon_2=\epsilon_2(M) \in (0,1]$ such that for any $\epsilon \in [0,\epsilon_2]$ and for any bi-Lipschitz change of variables $\Phi:\R^{n}\to\R^n$  satisfying $\|\nabla \Phi - I\|_\infty \leq \epsilon$, the map $\rho$ defined in \eqref{defrho} is a diffeomorphism between $\Omega_0$ and $\Omega:= \Phi(\Omega_0)$.
\end{theorem}

\bp
{\bf Step 1: We prove that $\rho(\Omega_0) \subset  \Omega$ when $\epsilon_2$ is sufficiently small.} Fix any $X\in\om_0$. 
First, we claim that 
\begin{equation} \label{claiml-hpi}
|\lambda(X) - h(\pi(X))| \leq Ct(X)\epsilon,
\end{equation}
where we recall that $h$ is the $C_M\epsilon_2$-Lipschitz function defined as $h(z) := \Phi(z,g(z)) - (z,g(z))$. 

Indeed, from the expression of $\lambda$ in \eqref{deflambda}, and since $\int \eta_t = 1$, we have
\begin{multline*}
|\lambda(X) - h(\pi(X))| = \int_{\R^{n-1}} |h(z) - h(\pi(X))| \eta_t(y-z) \, dz \lesssim \fint_{B(y,t)} |h(z) - h(\pi(X))|dz \\ \lesssim \epsilon \max\{ t, |y-\pi(X)|\}
\end{multline*}
since $h$ is $\epsilon$-Lipschitz. However, $|y-\pi(X)| = |\pi(X-tv_n) - \pi(X)| \leq t(X)$, so we have $|\lambda(X) - h(\pi(X))| \lesssim \epsilon t(X)$ as desired. 

We have then, by \eqref{defrho2}, that
\begin{equation} \label{claiml-hpi2}
|\rho(X) - X - h(\pi(X))| \leq |\lambda(X) - h(\pi(X)| + a t |\mathcal O - I_{n\times n}| + t |a-1| \lesssim \epsilon t(X)
\end{equation}
by \eqref{claiml-hpi}, $(i)$ of Lemma \ref{lemO}, and Lemma \ref{lemJ0}. 

We now look at the point $X+h(\pi(X))$. By the definition of $h$, we write 
\begin{multline}\label{eq.X+h}
    X+h(\pi(X))=X+\Phi\br{\pi(X),g(\pi(X))}-\br{\pi(X),g(\pi(X))}\\
    =\Phi(X)+\Phi\br{\pi(X),g(\pi(X))}-\Phi(X)-\br{\pi(X),g(\pi(X))}+X.
\end{multline}
We have
\begin{multline*}
    \abs{\Phi\br{\pi(X),g(\pi(X))}-\Phi(X)-\br{\pi(X),g(\pi(X))}+X}
    \le|X-\br{\pi(X),g(\pi(X))}|\norm{\nabla\Phi-I}_{\infty}\\
    \le \epsilon M \dist(X,\pom_0)
\end{multline*}
since $g$ is $M$-Lipschitz. This, together with \eqref{claiml-hpi2}, \eqref{eq.X+h} and the estimate $t(X)\approx\dist(X,\pom_0)$, gives that 
\[
|\rho(X)-\Phi(X)|\lesssim\epsilon\dist(X,\pom_0)\approx\epsilon\dist(\Phi(X),\pom),
\]
where in the last inequality we have used that the map $\Phi:\om_0\to\om$ is bi-Lipschitz. Therefore, by choosing $\epsilon_2$ sufficiently small, we have that 
\begin{equation}\label{eq.rho-Phi}
    |\rho(X)-\Phi(X)|\le\frac12\dist(\Phi(X),\pom),
\end{equation}
which implies that $\rho(X)\in\om$.
Step 1 follows. Note that \eqref{eq.rho-Phi} also implies that 
\begin{equation} \label{cclstep1}
\dist(\rho(X),\partial \Omega)  \approx \dist(\Phi(X),\partial \Omega)  \dist(X,\partial \Omega_0),
\end{equation}
where the second equivalence is due to the fact that $\Phi$ is bi-Lipschitz.

\medskip

\noindent {\bf Step 2: We show that $\rho$ is a diffeomorphism.}
Since $\rho$ is constructed from the Green function $G$, which is smooth on $\Omega_0$, the map $\rho$ is smooth on $\Omega_0$. Thus, it remains only to prove that $\rho$ is a bijection from $\Omega_0$ onto $\Omega$. 
We shall prove this in 3 steps. 

\medskip

\noindent {\bf (A) Local invertibility.}
By \eqref{defJ} and \eqref{defO},
\begin{equation}\label{defJac}
\nabla \rho
=
P J P^{-1}\mathcal O,
\end{equation}
where $P$ is the orthogonal matrix that changes the Cartesian base $\{e_1,\dots,e_n\}$ into $\{v_1,\dots,v_n\}$.
Theorem~\ref{ThrhoJ} and Lemma \ref{lemO} give that 
\[ \|J - I_{n\times n}\|_\infty + \|\mathcal O -  I_{n\times n}\|_\infty \lesssim \epsilon,\]
so if $\epsilon_2$ is small enough, 
\begin{equation}\label{Nrho-I<e}
\|\nabla\rho - I_{n\times n}\|_\infty \leq \frac12.
\end{equation}
It means that the map $\rho$ is locally invertible at any point. More precisely, for every $X\in\Omega_0$, there exist $r_X,r'_X>0$ such that
\begin{equation}\label{rholocbij}
\rho_\theta|_{B(X,r_X)}
:
B(X,r_X)\to \rho(B(X,r_X))
\supset B(\rho(X),r'_X)
\end{equation}
is a bijection.

\medskip

\noindent {\bf (B) Injectivity and Properness.}
For $X=(x,t)\in\Omega_0$, define the cones
\[
\Cone(X):=\{(z,s): |z-x|\leq M(s-t)\}.
\]
Since $g$ is $M$-Lipschitz, for $X,X'\in\om_0$,
\[
\Cone(X),\Cone(X')\subset \Omega_0.
\]
For $X,X'\in\om_0$, choose
\[
Y_X\in \Cone(X)\cap \Cone(X')
\]
with minimal height. The segments $[X_0,Y_X]$ and $[Y_X,X]$ lie in $\Omega_0$, so by \eqref{Nrho-I<e},
\begin{multline*}
|\rho(X)-X-\rho(X')+X'|\leq
\big(|X-Y_X|+|Y_X-X'|\big)
\|\nabla\rho-I\|_{L^\infty(\Omega)}
\\
\qquad\leq
C_M\varepsilon_2 |X-X'|
\leq \frac12 |X-X'|,
\end{multline*}
provided $\varepsilon_2$ is small enough. Consequently,
\begin{equation}\label{eq.rho_inj}
    |\rho(X)-\rho(X')|
\geq \frac12 |X-X'|\qquad\text{for }X,X'\in\om_0,
\end{equation}
which implies that $\rho$ is injective on $\om_0$. 
\smallskip

We claim that for every compact set $K\subset \om$,
\begin{equation}\label{rho-1(K)}
\rho^{-1}(K)
\text{ is compact in }\Omega_0.
\end{equation}
Indeed, by \eqref{cclstep1},
preimages of compact subsets of $\om$ stay away from $\partial\Omega_0$. Moreover, since $\rho$ is smooth and so in particular continuous, $\rho^{-1}(K)$ is closed. It therefore suffices to prove that $\rho^{-1}(K)$ is bounded. By fixing $X'$ in \eqref{eq.rho_inj}, one sees that 
\begin{equation}\label{eq.rho_prop}
    |\rho(X)|\to\infty \quad\text{as }|X|\to\infty.
\end{equation} 
It means that $\rho^{-1}(K)$ cannot be unbounded, as if it were unbounded, then there would exist a sequence of points $(X_n)_{n\in\mathbb{N}}\subset\rho^{-1}(K)$ such that $|X_n|\to\infty$ as $n\to\infty$. But \eqref{eq.rho_prop} asserts that $|\rho(X_n)|\to\infty$, contradicting to the assumption that $K$ is compact. So the claim \eqref{rho-1(K)} follows.
\medskip

\noindent {\bf (C) Surjectivity.}
We now show that $\rho:\om_0\to\om$ is surjective, and thus a diffeomorphism. For $\mathcal Z\in \om_0$, define the degree
\[
N_{\rho}(\mathcal Z)
:=
\#\{X\in\Omega_0:\rho(X)=\mathcal Z\}.
\]
We shall prove that
\[
N_\rho(\mathcal Z)=1
\qquad\text{for every }\mathcal Z\in\Omega,
\]
which implies that $\rho$ is a diffeomorphism.

We know from {\bf(B)} that $\rho$ is injective, and thus $N_\rho(\rho(X)) = 1$. So, we can write 
\[\Omega := \underbrace{\{\mathcal Z \in \Omega, \, N_\rho(\mathcal Z) = 1\}}_{S_1} \cup \{\mathcal Z \in \Omega, \, N_\rho(\mathcal Z) = 0\}.\]
In order to prove that $S_1 = \Omega$, we then just have to show that $S_1$ is open and closed in $\Omega$.

\smallskip

\noindent \underline{$S_1$ is open.} If $\mathcal Z \in S_1$, then there exists $X\in \Omega_0$ such that $\rho(X) = \mathcal Z$, and \eqref{rholocbij} shows that a neighborhood of $\mathcal Z$ is contained in the range by $\rho$ of a neighborhood of $X$. Hence $S_1$ is open.

\noindent \underline{$S_1$ is closed.}  Let $(\Z_k)_{k\in\bN}$ be a sequence in $S_1$ such that $\Z_k\to \Z$ for some $\Z$ as $k\to\infty$. Then we get a sequence $\{X_k\}_{k\in \bN}$ in $\Omega_0$ such that $\rho(X_k) \to \mathcal Z$ as $k\to\infty$. By \eqref{rho-1(K)}, we know that $\{X_k\}_{k\in \bN}$ is contained in a compact subset of $\Omega_0$, hence up to taking a subsequence, we can assume that $X_k$ converges to some $X_\infty\in \Omega_0$. The continuity of $\rho$ gives that $\rho(X_\infty) = \mathcal Z$. We conclude that $\mathcal Z \in S_1$, which implies that $S_1$ is closed.

\smallskip

To summarize, we showed that $\Omega = S_1 = \{\mathcal Z \in \Omega, \, N_\rho(\mathcal Z) = 1\}$, which means that the local diffeomorphism $\rho$ is a (global) diffeomorphism from $\Omega_0$ to $\Omega$. The theorem follows.
\ep

\section{Proof of the main theorem and corollary}

\subsection{Proof of Theorem \ref{MainTh}}

When $\rho$ is a bi-Lipschitz change of variables from $\om_0$ to $\om$, we say that $L_\rho$ is the conjugate of $-\Delta$ by $\rho$ if for any harmonic function $u$ in $\Omega$, $u_\rho:= u\circ \rho$ is a weak solution to $L_\rho u_\rho = 0$ in $\Omega_0$.

It is a simple exercise to see (see for instance (6.21) in \cite{DFMDahlberg}) that $L_\rho:= -\diver A_\rho \nabla$ with
\begin{equation} \label{defArho}
A_\rho(X) := \det(\nabla\rho(X)) (\nabla\rho)^{-T}(X) (\nabla\rho)^{-1}(X)
\end{equation}
where $\nabla\rho$ is the Jacobian matrix of $\rho$. From there, we have the following result:

\begin{lemma} \label{lemLrho}
Let $\Omega_0:=\{ (x,t) \in \R^n,\, t> g(x)\}$, where $g$ is a $M$-Lipschitz function, and write $G$ for the Green function with pole at infinity for $-\Delta$ on $\om_0$. There exist constants $\epsilon_2=\epsilon_2(n,M) \in (0,1]$ and $C_2=C_2(n,M)>0$ such that the following holds. 

For any $\epsilon \in [0,\epsilon_2]$, take any bi-Lipschitz map $\Phi:\R^{n}\to\R^n$  satisfying $\|\nabla \Phi - I\|_\infty \leq \epsilon$, and set $\Omega = \Phi(\Omega_0)$.
Let $\rho=\rho_{\Phi}$ be the map defined in \eqref{defrho}. Let $L_\rho$ be the conjugate of $-\Delta$ by $\rho$. There exists a symmetric operator $L_0 := -\diver A_0 \nabla$ on $\Omega_0$ such that
\begin{enumerate}[(i)]
\item for all $\xi \in \R^n$ and $X\in \Omega_0$
\begin{equation} \label{ellipA0}
\frac12 |\xi|^2 \leq A_0(X)\xi\cdot \xi \leq 2 |\xi|^2;
\end{equation}
\item $\|A_0 - I\|_\infty \leq C_2 \epsilon$;
\item $|A_\rho - A_0| \in \CMs(C_2\epsilon)$, where the Carleson condition is with respect to domain $\om_0$;
\item $L_0 G = 0$ in $\om_0$.
\end{enumerate}
\end{lemma}

\bp
Take $\epsilon_2$ small enough such that Theorems \ref{ThrhoJ} and \ref{Thrhodiffeo} apply. Let $P=P(X)$ be the orthogonal matrix that changes the base $\{e_1,\dots,e_n\}$ into $\{v_1,\dots,v_n\}$. Recall that $\nabla\rho$ is given by 
\[\nabla\rho = P J P^{-1} \mathcal O\]
as in \eqref{defJac}. So the matrix $A_\rho$ is then given by
\[A_\rho(X) = \det(J) P J^{-T} J^{-1} P^{-1}, \]
because the matrices $P$ and $\mathcal O$ are orthogonal.

\medskip

The choice of $A_0$---hence of $L_0$---is given by
\[A_0:= \det(J_0) P (J_0)^{-T} (J_0)^{-1} P^{-1}.\]
From the construction and the fact that $P$ is an orthogonal matrix, it follows that $A_0$ is symmetric.  
Recall that $\|J_0 - I\|_\infty \leq C\epsilon$ by Lemma  \ref{lemJ0}. So---assuming that $\epsilon_2$ is small enough---the matrix $A_0$ is elliptic, and since the  determinant is a multilinear form, we have 
\[\|\det(J_0) - 1\|_\infty \leq C\epsilon.\]
By writing 
\[
A_0 - I = (\det(J_0) - \det(I))PJ_0^{-T}J_0^{-1}P^{-1}+P(J_0^{-1}+I)J_0^{-1}(I-J_0)P^{-1},
\]
the above estimates give that 
\begin{multline*}
    \|A_0 - I\|_\infty \leq \|\det(J_0) - \det(I)\|_\infty \|P (J_0)^{-T} (J_0)^{-1} P^{-1}\|_\infty \\ 
    + \|P\|_\infty\|P^{-1}\|_\infty \|(J_0)^{-1}\|_\infty \big(\|I\|_\infty + \|(J_0)^{-1}\|_\infty\big) \|J_0- I\|_\infty \leq C \epsilon
\end{multline*}
because all the considered matrices are uniformly bounded in $X \in \Omega_0$. So by choosing $\epsilon_2$ small enough, we can guarantee that $A_0$ satisfies the ellipticity condition \eqref{ellipA0}.

Since we also have that $J-J_0 \in \CMs(C\epsilon)$ and $\|J-J_0\|_\infty \leq C\epsilon$ by Lemma \ref{ThrhoJ}, the multilinearity of the determinant also gives that
\[|\det(J) - \det(J_0)| \in \CMs(C\epsilon).\]
We conclude that 
\begin{multline*}
    |A_\rho - A_0| \leq |\det(J) - \det(J_0)| \|P J^{-T} J^{-1} P^{-1}\|_\infty \\ 
    + \|P\|_\infty\|P^{-1}\|_\infty \|\det(J_0)\|_\infty \|J^{-1}\|_\infty \|(J_0)^{-1}\|_\infty \big(\|J^{-1}\|_\infty + \|(J_0)^{-1}\|_\infty\big) |J_0- J| \\ \in \CMs(C\epsilon).
\end{multline*}

\medskip

It remains to prove that $L_0 G = 0$. This is true because of our careful construction of $\rho$. More precisely, we constructed $\rho$ to preserve the structure of $\Omega_0$, so that $G$ is close to be a solution to $L_\rho u = 0$.

We claim that for $j\in\set{1,\dots,n}$,
\begin{equation} \label{claimA0vn}
    \left< A_0 v_n^T,v_j\right> = I_{jn}.
\end{equation}
Indeed, $\left< A_0 v_n^T,v_j\right>$ is the coefficients of the last column of $A_0$ in the base $\{v_1,\dots,v_n\}$. The matrix $A_0$ in the base $\{v_1,\dots,v_n\}$ is given by
\[A_{0,v} := P^{-1} A_0 P = \det(J_0) (J_0)^{-T} (J_0)^{-1}.\]
However, the matrix $J_0$ - defined in \eqref{blockJ0} - is diagonal by block, so it is easy to compute the last column of $A_{0,v}$: it is given by
\[(0,\dots,0,\det(J_0) a^{-2})^T = (0,\dots,0,1)^T\]
since $\det(J_0) = a \det(A_{||}) = a^2$. 

The claim \eqref{claimA0vn} gives that $A_0 v_n^T = v_n^T$, which implies that 
\[A_0 \nabla G = \nabla G\]
since $v_n = \nabla G^T/|\nabla G|$. It means that 
\[L_0 G = -\diver A_0 \nabla G = -\Delta G = 0 \quad \text{in }\om_0,\]
since $G$ is the Green function (for $-\Delta$) with pole at infinity. The lemma follows.
\ep

\bigskip

\noindent{\em Proof of Theorem \ref{MainTh}.} Let $\Omega_0 \subset \R^n$ and $p\in (1,\infty)$ be as in the theorem. Let $G_{-\Delta}$ and $\omega_{-\Delta}$ be the Green function and harmonic measure with pole at infinity for $\om_0$ defined in Lemma~\ref{lem GrEm_infty}. Since the $L^p$ Dirichlet problem for $-\Delta$ is solvable in $\Omega_0$,  there is a constant $C_p>0$ such that the reverse H\"older inequality \eqref{defRH} holds.
\begin{itemize}
\item Let $\epsilon_1= \epsilon_1(n,p,C_p,M,2)$ be the constant given by Theorem \ref{ThCarlpert}, that is the constant under which Carleson perturbations preserves the solvability of the Dirichlet problem in $L^p$ if as long as the operator in $\Omega_0$ before perturbation verifies the reverse H\"older bound \eqref{defRH} and is uniformly elliptic with constant $2$.

\item Let $\epsilon_2 = \epsilon(n,M)$ and $C_2=C_2(n,M)$ be the constants given by Lemma \ref{lemLrho}. 
\end{itemize}
Define
\[\epsilon_0:= \min\{\epsilon_2,\epsilon_1/C_2\}.\]
Take then a bi-Lipschitz change of variables $\Phi$ verifying $\|\nabla \Phi - I\|_\infty \leq \epsilon_2$ and set $\Omega:= \Phi(\Omega_0)$. We can construct the bi-Lipschitz change of variables $\rho$ as in \eqref{defrho} that maps $\Omega_0$ to $\Omega$, and then the conjugate $L_\rho$ of $-\Delta$ by $\rho$ as in \eqref{defArho}. Take then $L_0$ as in Lemma \ref{lemLrho}.

\medskip

With all the necessary tools being introduced, the argument is then straightforward. 
\begin{enumerate}
    \item Since $L_0 G_{-\Delta} = -\Delta G_{-\Delta} = 0$ in $\om_0$, one has $G_{-\Delta} = G_{L_0}$, that is, $G_{-\Delta}$ is the Green function with pole at infinity for $L_0$ in $\Omega_0$. It means in particular that $L_0$ satisfies the reverse H\"older inequality \eqref{defRH} with the same constant $C_p$ as the Laplacian.
    \item By Lemma \ref{lemLrho}, we have that $L_0$ is uniformly elliptic with constant 2, and $|A_\rho-A_0| \in \CMs(C_2\epsilon_0)$. But since $C_2\epsilon_0 \leq \epsilon_1$, Theorem \ref{ThCarlpert} gives that the Dirichlet problem (for $L_\rho$ in $\Omega_0$) is also solvable for boundary data in $L^p$.
    \item Since a bi-Lipschitz change of variables preserves the solvability of the $L^p$ Dirichlet problem, see Proposition~\ref{prop.DbiLip}, the $L^p$ Dirichlet problem for $-\Delta$ in $\Omega$ is also solvable.
\end{enumerate}
Succinctly, the proof is

\noindent  \begin{tikzcd}
(D_{p})_{-\Delta,\Omega_0} \arrow[rrr,Leftrightarrow, bend left, "{\tiny G_{-\Delta} = G_{L_0}}" above, "{\text{\tiny \begin{tabular}{c} Lem \ref{lemLrho} (iii) \\ $+$ Thm \ref{ThDP<=>RH}\end{tabular}}}" below] 
 & & & (D_{p})_{L_0,\Omega_0} \arrow[rrr,Rightarrow, bend left, "{\tiny A_0 - A_\rho \in \CMs(\epsilon_1)}" above, "{\text{\tiny \begin{tabular}{l} Thm \ref{ThCarlpert} \end{tabular}}}" below] 
 & & & (D_{p})_{L_\rho,\Omega_0} \arrow[rrr,Leftrightarrow, bend left, "{\tiny \rho \text{ bi-Lipschitz}}" above, "{\text{\tiny \begin{tabular}{l} Prop \ref{prop.DbiLip} \end{tabular}} }" below] 
 & & & (D_{p})_{-\Delta,\Omega}
\end{tikzcd}

\noindent Theorem \ref{MainTh} follows. \ep

\subsection{Proof of Corollary \ref{MainCor}}\label{S.pfCor}

Let $\Omega \subset \R^n$ be a bounded strongly quasiconvex domain. Take $X_0$ to be a center of $\Omega$, that is, a point such that $\dist(X_0,\partial \Omega) \approx \diam(\Omega)$, and set $M$ to be the Lipschitz constant that appears in Definition \ref{defquasiconvex}. Let $G^{X_0}$ and $\omega^{X_0}$ be, respectively, the Green function and harmonic measure with pole at $X_0$ for $\om$.

Let $p\in (1,\infty)$. Let $C_p = C_p(M)$ be the reverse H\"older constant which is uniform to $M$-Lipschitz convex domains. Let $\epsilon_0:= \epsilon_0(n,p,C_p,M) >0$ be as in Theorem \ref{MainTh}, and then $r_0>0$ as in Definition \ref{defquasiconvex}. Without loss of generality, we can assume that $r_0<\frac{\dist(X_0,\pom)}{2}$. 

\medskip

We will prove that there exists $m\in \mathbb N$ and $C>0$ such that for any $x_0 \in \partial \Omega$ and any $r\in (0,2^{-m}r_0]$, we have the reverse H\"older inequality
\begin{equation} \label{claimMainCor}
\left( \fint_{B(x_0,r)\cap \partial \Omega} \br{\kappa^{X_0}}^{p'} \, d\sigma\right)^\frac1{p'} \leq C \fint_{B(x_0,r)\cap \partial \Omega} \kappa^{X_0} \, d\sigma,
\end{equation}
where $\kappa^{X_0}(x):=\limsup_{X\in \gamma(x)\atop X\to x} \frac{G^{X_0}(X)}{\delta(X)}$ is defined as in \eqref{defRH'}.
The claim \eqref{claimMainCor} is sufficient to prove Corollary \ref{MainCor}. Indeed, by Theorem \ref{ThDP<=>RH}, the Corollary would follow if we had \eqref{claimMainCor} for all $r\in (0,\diam \Omega)$, instead of of only $r\in (0,2^{-m}r_0]$. However, the reverse H\"older bounds for $r\in (2^{-m}r_0,\diam \Omega)$ can easily deduced from the reverse H\"older bounds for $r=2^{-m}r_0$ and a covering argument.

\medskip

Take $x_0 \in \partial \Omega$. Let $g=g_{x_0}$ be the $M$-Lipschitz function and $\Phi=\Phi_{x_0}$ be the bi-Lipschitz change of variables given by Definition \ref{defquasiconvex} (for a coordinate system that also depend on $x_0$), that is such that 
\[\Omega \cap B(x_0,r_0) = \underbrace{\Phi_{x_0}(\{(x,t)\in \R^n, \, t>g_{x_0}(x)\})}_{=:\Omega_{x_0}} \cap B(x_0,r_0).\]
We write $\sigma_{x_0}$ for the surface measure on $\partial \Omega_{x_0}$, $G_{x_0}$ for the Green function for $-\Delta$ with pole at infinity in the domain $\Omega_{x_0}$ normalized so that 
\begin{equation}\label{Gx0norm}
    G_{x_0}(X_{x_0,m})=G^{X_0}(X_{x_0,m}),
\end{equation}
where $X_{x_0,m}\in B(x_0,r_0/2)\cap\om$ satisfies \[\dist(X_{x_0,m},\pom)=\dist(X_{x_0,m},\pom_{x_0})\approx |X_{x_0,m}-x_0|\approx 2^{-m}r_0,\]
$m\ge1$ to be determined. 
We write $\kappa_{x_0}$ for
\[\kappa_{x_0}(x):= \limsup_{X\in \gamma_{x_0}(x) \atop X \to x} \frac{G_{x_0}(X)}{\dist(X,\partial \Omega_{x_0})}, \qquad x\in \partial \Omega_{x_0}.\]
Here, $\gamma_{x_0}(x):= \{Z\in \Omega_{x_0}, \, |Z-x| \leq 2 \dist(Z,\partial \Omega_{x_0})$ as expected.

By definition, the domain $\om_{0,x_0}:=\set{(x,t)\in\Rn: \, t>g_{x_0}(x)}$ is a convex Lipschitz graph domain, and so Theorem~\ref{thm-Dp-C1a} applies. 
It means that the $L^p$ Dirichlet problem for $-\Delta$ is solvable in $\Omega_{0,x_0}$, and that the constant $C_p$ in \eqref{defRH} is independent of $x_0$ even though $\om_{0,x_0}$ itself depends on $x_0$. Therefore, by Theorem \ref{MainTh}, the $L^p$ Dirichlet problem for $-\Delta$ is solvable in $\Omega_{x_0}$. Then Theorem \ref{ThDP<=>RH} entails that there exists a constant $C'_p:=C'_p(n,p,C_p,M)$---so independent of $x_0$---such that we have the reverse H\"older bound
\begin{equation} \label{RHforCor1} 
\left( \fint_{B(x,r)\cap \partial \Omega_{x_0}} |\kappa_{x_0}|^{p'} \, d\sigma \right)^\frac{1}{p'} 
\leq C'_p \fint_{B(x,r)\cap \partial \Omega_{x_0}}|\kappa_{x_0}| \, d\sigma  \qquad \text{ for all } x\in \partial \Omega_{x_0}, \, r>0.
\end{equation}
This last bound is not too far from our objective \eqref{claimMainCor}, we just need to drop the dependence of $\kappa$ in $x_0$. We do this using the comparison principle.

Observe first that for $x$ and $r$ such that $B(x,r)\subset B(x_0,r_0)$, one has $B(x,r)\cap\pom_{x_0}=B(x,r)\cap\pom$, and that  
\begin{equation} \label{truc1}
    \dist(Z,\partial \Omega_{x_0}) = \dist(Z,\partial \Omega) \quad \text{ whenever $Z\in B(x_0,3r_0/2)$, $\dist(Z,\partial \Omega)<r_0/10$.}  
\end{equation}
Moreover, since $G^{X_0}$ and $G_{x_0}$ are both solutions to $-\Delta u = 0$ in $B(x_0,r_0) \cap \Omega$, we have by the comparison principle and the H\"older continuity of solutions (see e.g. \cite[Corollary 6.4]{DEM}) that
\begin{equation*}
    \left|\frac{G_{x_0}(X)G^{X_0}(Y)}{G^{X_0}(X) G_{x_0}(Y)} - 1 \right| \leq C \br{\frac{\rho}{r_0}}^\alpha \quad \text{ whenever $X,Y\in B(x_0,\rho)\cap\om$},\quad \rho<r_0/2,
\end{equation*}
with constants $C$, $\alpha$ depending only on $n$ and $M$. We now take $Y=X_{x_0,m}$ and choose $m$ large enough so that 
\begin{equation} \label{truc2}
     \left|\frac{G_{x_0}(X)}{G^{X_0}(X)} - 1 \right| \leq C2^{-m\alpha}\le \frac12 \qquad\text{for }X\in B(x_0,2^{-m}r_0)\cap \om,
\end{equation}
where we have used the normalization \eqref{Gx0norm}. 

The combination of \eqref{truc1} and \eqref{truc2} gives that 
\[ \frac12 \kappa(x)\le \kappa_{x_0}(x) \le \frac32 \kappa(x) \qquad \text{ for } x\in B(x_0,2^{-m}r_0)\cap \partial \Omega.\]
From this and \eqref{RHforCor1},  \eqref{claimMainCor} follows when $r \leq 2^{-m}r_0$, and so does Corollary \ref{MainCor}.

\appendix
\section{Proof of Remarks~\ref{rm.Om-graph} and \ref{rm.C1invariant}}\label{S.App-rmk}
In this appendix, we prove the two observations made in Remarks~\ref{rm.Om-graph} and \ref{rm.C1invariant}.
\begin{proof}[Proof of Remark~\ref{rm.Om-graph}]
We write the inverse map of $\Phi$ by  $\Phi^{-1}(x,t) = (\phi(x,t), T(x,t))$ for $(x,t)\in\R^{n-1}\times\R$, where $\phi:\Rn\to\R^{n-1}$ and $T:\Rn\to\R$. Then if $(x,t)\in\pom$, it follows that $\Phi^{-1}(x,t)\in\pom_0$, and thus $T(x,t) = g(\phi(x,t))$. Since $|\nabla \Phi - I| \le \epsilon_0$, we have $|\nabla \Phi^{-1} - I| \lesssim \epsilon_0$. This implies that we can write $\phi(x,t) = x + \psi(x,t)$ and $T(x,t) = t - h(x,t)$ for some mapping $\tilde\psi:\Rn\to\R^{n-1}$ and $\tilde h:\Rn\to\R$ that satisfy $|\nabla \tilde \psi| + |\nabla\tilde h| \lesssim \epsilon_0$. Hence, for $(x,t)\in\pom$, there holds
     $$t = g(x + \psi(x,t)) + h(x,t).$$
    For $\epsilon_0>0$ small enough so that $C\epsilon_0(M+1)<1$, we can apply the implicit function theorem to solve the above equation and get that $t = f(x)$ for some Lipschitz function $f:\R^{n-1}\to\R$. Hence, $\pom$ can be expressed as
     $$\set{(x,t)\in\R^{n-1}\times\R: t = g(x + \tilde\psi(x,f(x))) + \tilde h(x,f(x))}. $$
     We let $\psi(x)=\tilde\psi(x,f(x))$, $h(x)=\tilde h(x,f(x))$, and check that they have small gradients as desired.
\end{proof}

\begin{proof}[Proof of Remark~\ref{rm.C1invariant}]
Let $\Omega$ be a strongly quasiconvex domain and let $x_0 = 0 \in \partial \Omega$. By definition, for any $\epsilon_0>0$, there exists $r>0$, a bi-Lipschitz map $\Phi$ with $\norm{\nabla\Phi-I}\le \epsilon_0$, and an $M$-Lipschitz function $g$ such that $\om\cap B(0,r)=\Phi(\{ t>g(x)\})  \cap B(0,r)$. Let $\Psi:\Rn\to\Rn$ be a $C^1$ diffeomorphism. We want to show that $\Psi(\om)$ is strongly quasiconvex. 

We can assume without loss of generality that $\Psi(0) = 0 = \Phi(0)$. For some $r_0\in(0,r)$, we have
     $$ \Psi(\Omega) \cap B(0,r_0) = \Psi\circ \Phi (\{ t>g(x) \})  \cap B(0,r_0).$$
     Define a nondegenerate linear transform $L:\Rn\to\Rn$, $L(X) = \nabla \Psi(0)X  = X_i\, \partial_i \Psi_j(0)$. 
     Write
     $$ 
     \Psi\circ \Phi(X) = \Psi\circ \Phi \circ (L^{-1} (L (X))). 
     $$ 
     Note that the $L$ transforms convex domains into convex domains. Thus up to a change of coordinates, $L(\{ t > g(x) \}) = \{ t > \hat{g}(x) \}$, where $\hat{g}$ is also a convex function. Hence
     $$ 
     \Psi(\Omega) \cap B(0,r_0) = \Psi\circ \Phi \circ L^{-1} ( \{ t > \hat{g}(x) \}). 
     $$ 
     Finally note that $\Phi(X) = X + O(\epsilon_0)|X|$ is close to the identity transform and that $L^{-1}$ is close to $\Psi^{-1}$ in a neighborhood of $0$ since $\Psi(X) = LX + o(|X|)$ (using the fact that $\Psi$ is $C^1$). Thus by the chain rule, we can  show that $|\nabla(\Psi\circ \Phi \circ L^{-1}) - I| \le C\epsilon_0$ if $r_0$ is small enough. This means that $\Psi(\Omega)$ is strongly quasiconvex.
\end{proof}

\section{An alternative proof of Lemma~\ref{lemSL=>H}}\label{S.altpf} We can assume without loss of generality that $g(0)=0$, that is, $0\in\pom_0$. 

\medskip

{\bf Step 1: Show that $\partial_n G\geq 0$ in $\om_0$.} This step carries all the key arguments, and \eqref{H1}--\eqref{H2} is actually a corollary of this statement.

\medskip

\noindent {\bf Step 1a: Convergence to $G$.} For $Y\in\om_0$, let $G^Y$ be a smooth version of Green function with pole at $Y$. Specifically, if $G(X,Y)$ is the Green function in $\Omega_0$, and if $\eta_t$ is a mollifier, meaning that $\eta \in C^\infty(B_{\R^n}(0,1/2))$, $\eta \geq 0$, $\int_{\R^n} \eta = 1$, and $\eta_t(Z) := t^{-n} \eta(Z/t)$, then 
\[G^Y(X) := \int_{\Omega_0} \eta_{\delta(Y)}(Y-Z) G(X,Z) \, dZ.\]
Let $Y_i$ be the point $(0,2^{i})$. From the construction of the Green function with pole at infinity, and in particular, \eqref{Greeninfty_lmt}, we have that
\begin{equation} \label{convGYtoGwZ}
\frac{G(.,Z_i)}{G(Y_0,Z_i)} \text{ converges to } \frac{G}{G(Y_0)} \text{ uniformly on compacts of $\overline{\Omega_0}$}.
\end{equation}
as $i\to \infty$, as long as $Z_i \in B(Y_i,\delta(Y_i)/2)$. Moreover, the rate of convergence does not depend on the choice of $Z_i \in B(Y_i,\delta(Y_i)/2)$. As a consequence
\begin{equation} \label{convGYtoG}
\frac{G^{Y_i}}{G^{Y_i}(Y_0)} \text{ converges to } \frac{G}{G(Y_0)} \text{ uniformly on compacts of $\overline{\Omega_0}$}.
\end{equation}

The De Giorgi-Nash-Moser estimate and the Caccioppoli inequality show then the convergence of the gradient, that is
\begin{equation} \label{convdtGYtodtG}
\frac{\partial_n G^{Y_i}}{G^{Y_i}(Y_0)} \text{ converges to } \frac{\partial_n G}{G(Y_0)} \text{ uniformly on compacts of $\Omega_0$}.
\end{equation}
Indeed, take $K\Subset \Omega_0$ be compact, and chose $i\in \bN$ large enough such that $|Y_i| \geq 4(\diam K + \dist(K,\partial \Omega_0)$. Then for any $X\in K$, since both $G^{Y_i}$ and $G$ are solutions to $\Delta u = 0$ in $B(X,\delta(X))$, so is $\frac{\partial_n G^{Y_i}}{G^{Y_i}(Y_0)} - \frac{\partial_n G}{G(Y_0)}$.
Therefore, by the De Giorgi-Nash-Moser estimate and the Caccioppoli inequality, we have
\begin{multline*}
   \left| \frac{\partial_n G^{Y_i}(X)}{G^{Y_i}(Y_0)} - \frac{\partial_n G(X)}{G(Y_0)} \right| 
   \lesssim \left( \fint_{B_X/2} \left| \frac{\partial_n G^{Y_i}(Z)}{G^{Y_i}(Y_0)} - \frac{\partial_n G(Z)}{G(Y_0)} \right|^2 \, dZ \right)^\frac12 \\
   \lesssim \frac1{\dist(K,\partial \Omega_0)} \left( \fint_{B_X} \left| \frac{G^{Y_i}(Z)}{G^{Y_i}(Y_0)} - \frac{G^{Y_i}(Z)}{G^{Y_i}(Y_0)} \right|^2 \, dZ \right)^\frac12 \\
   \lesssim \frac1{\dist(K,\partial \Omega_0)} \sup_{Z\in K^*} \left| \frac{G^{Y_i}(Z)}{G^{Y_i}(Y_0)} - \frac{G(Z)}{G(Y_0)} \right| \longrightarrow 0 \text{ as } i \to \infty.
\end{multline*}
where $K^*:= \bigcup_{Z\in K} \overline{B(Z,\delta(Z)/2)}$, and the final convergence is by \eqref{convGYtoG}.

\medskip

\noindent {\bf Step 1b: Representation of $\partial_n G^{Y}$.}
We write $H^{Y}$ for the function
\[H^Y(X):= \int_{\Omega_0} \br{\partial_n \eta_{\delta(Y)}}(Y-Z) G(X,Z) dZ.\]
We claim that 
\begin{equation}\label{claim.drnG+H}
    \partial_n G + H^Y  \text{ is a weak solution to $-\Delta u = 0$ in $\Omega_0$.}
\end{equation}
First, $G^Y$ and $H^Y$ are smooth functions, since there both weak solution to $-\Delta u = f$ with some $f \in C^\infty_0(\Omega)$. More precisely, $-\Delta G^Y=\eta_{\delta(Y)}(Y-\cdot)$, and $-\Delta H^Y=\br{\partial_n \eta_{\delta(Y)}}(Y-\cdot)$.
Hence, $\partial_n G^Y$ and $H^Y$ are in $W^{1,2}_{loc}(\Omega_0)$. Second, for $\varphi \in C^\infty_0(\Omega)$, we have by the PDE satisfied by $ G^Y$ and $H^Y$,
\begin{multline*}
    \int_{\Omega_0} \nabla [\partial_n G^Y + H^Y] \cdot \nabla \varphi \, dZ
    = - \int_{\Omega_0} \nabla G^Y \cdot \nabla \partial_n \varphi \, dZ 
    + \int_{\Omega_0} \br{\partial_n \eta_{\delta(Y)}}(Y-Z) \varphi(Z)\, dZ \\
    =  \int_{\Omega_0} \big( - \eta_{\delta(Y)}(Y-Z) \partial_n \varphi(Z) 
    + \br{\partial_n \eta_{\delta(Y)}}(Y-Z) \varphi(Z)\big) \, dZ.
\end{multline*}
Integration by parts in $Z_n$ gives that 
\[
 -\int_{\Omega_0}\eta_{\delta(Y)}(Y-Z) \partial_n \varphi(Z)dZ
 =\int_{\om_0}\dr_{Z_n}\br{\eta_{\delta(Y)}(Y-Z)}\vp(Z)dZ=-\int_{\om_0}(\dr_n\eta_{\delta(Y)})(Y-Z)\vp(Z)dZ,
\]
which implies that $ \int_{\Omega_0} \nabla [\partial_n G^Y + H^Y]\cdot\nabla\vp=0$ as claimed.

\medskip

Since $\Omega_0$ is a Lipschitz domain, the $L^2$ Dirichlet problem is solvable for $-\Delta$ in $\om_0$ (see \cite{Dah77} or \cite{JK81}\footnote{An attentive reader might notice that both references are for bounded Lipschitz domain, but the same result holds for unbounded domain. Alternatively, since we actually only need $L^p$ estimate for some $p>1$, we can apply \cite{HKMP} for unbounded domains.}). So, by Theorem 1.22 in \cite{MPT}, the Poisson-Dirichlet and Poisson-regularity problem are also solvable in $L^2$. It implies that 
\[\|N(\nabla G^Y)\|_{L^2(\pom_0)} + \|N(H^Y)\|_{L^2(\pom_0)} < +\infty.\]
In particular, $\partial_n G^Y$ has a trace in $L^2(\pom_0)$. The $L^2$ solvability of the Dirichlet problem in $\Omega_0$ also implies that the harmonic extension $u_{\partial_n G^Y}$ of this trace in $\Omega_0$ verifies
\[\|N(u_{\partial_n G^Y})\|_{L^2(\pom_0)}\lesssim\norm{Tr(\dr_n G^Y)}_{L^2(\pom_0)} < +\infty.\]
Note that the non-tangential bounds of $\partial_n G^Y$, $H^Y$, and $u_{\partial_n G^Y}$ guarantees the almost everywhere non-tangential limit of the functions to their traces, and that 
\[\Tr(\partial_n G^Y + H^Y - u_{\partial_n G^Y}) = \Tr(H^Y) = 0.\]
Consequently, $\partial_n G^Y + H^Y - u_{\partial_n G^Y}$ is a function with non-tangential bounds and zero trace, and is a solution to $-\Delta u=0$ in $\om_0$ by \eqref{claim.drnG+H}. Uniqueness of the Dirichlet problem (see \cite[Theorem 1.4.4]{Ken94}, or \cite[Section 6]{HLMP} for unbounded domains) gives the formula: 
\[\partial_n G^Y = u_{\partial_n G^Y} - H^Y \quad \text{in }\om_0.\]

\medskip

\noindent {\bf Step 1c: $\partial_n G \geq 0$.} First, note that $\partial_n G$ is non-negative on the boundary $\partial \Omega_0$. Indeed, for $x\in \partial \Omega$, since $G^Y$ is positive in $\Omega_0$ and since $x+he_n \in \Omega_0$ for $h>0$, we have
\[\partial_n G^Y(x) = \lim_{h\to 0 \atop h>0} \frac{G^Y(x+he_n)}{h} \geq 0.\]
As a consequence, the harmonic extension $u_{\partial_n G^Y}$ of $\Tr(\partial_n G^Y)$ - which is given by the integral of $\Tr(\partial_n G^Y)$ with respect to a positive measure (the harmonic measure) - is non-negative in $\Omega_0$.

\medskip

Second, we compare $H^Y$ and $G^Y$ outside of $B_Y$. 
Observe that 
\[
H^Y(X)=-\int_{B_Y}\dr_{Z_n}\br{\eta_{\delta(Y)}(Y-Z)}G(X,Z)dZ=\int_{B_Y}\eta_{\delta(Y)}(Y-Z)\dr_{Z_n}G(X,Z)dZ
\]
by integration by parts. Therefore, 
\[|H^Y(X)|\le\int_{B_Y}\eta_{\delta(Y)}(Y-Z)|\nabla_{Z}G(X,Z)|dZ.
\]
If $B_X$ and $B_Y$ do not intersect, we apply Lemma \ref{lemNG<G/d} to $G(X,\cdot)$ in $\om_0\setminus B_X$. For $Z\in B_Y$, $\dist(Z,\dr(\om_0\setminus B_X))=\min\set{\delta(Z),\dist(Z,B_X)}$. Obviously, $\delta(Z)\ge \delta(Y)/2$. We check that we have $\dist(Z,B_X)\ge \max\set{\delta(X)/2,|Y-X|-\delta(X)/2-\delta(Y)/2}$. If $\delta(X)<\delta(Y)/4$, then $|Y-X|\ge \delta(Y)-\delta(X)$, and so  $|Y-X|-\delta(X)/2-\delta(Y)/2\ge \delta(Y)/8$. This implies that $\dist(Z,B_X)\ge \delta(Y)/8$, and thus for $Z\in B_Y$, $\dist(Z,\dr(\om_0\setminus B_X))\ge\delta(Y)/8$. So we have that
\[|H^Y(X)|\lesssim  \int_{B_Y}\eta_{\delta(Y)}(Y-Z)\frac{G(X,Z)}{\dist(Z,\dr(\om_0\setminus B_X))}dZ\lesssim  \int_{B_Y}\eta_{\delta(Y)}(Y-Z)\frac{G(X,Z)}{\delta(Y)}dZ = \frac{G^Y(X)}{\delta(Y)}.\]
It means that for any $X\in \Omega_0$ and any $c>0$, there exists $i_0:=i_0(X,c)$ such that for $i\geq i_0$
\begin{equation} \label{H<cG}
|H^{Y_i}(X)| \lesssim 2^{-i} G^{Y_i}(X) \lesssim c G^{Y_i}.
\end{equation}
For $c>0$, using successively Step 1a, Step 1b, the positivity of $u_{\partial_n G^{Y_i}}$ and \eqref{H<cG}  we have then that
\[ \frac{\partial_n G(X)}{G(X)} = \lim_{i\to \infty} \frac{\partial_n G^{Y_i}(X)}{G^{Y_i}(X)} = \lim_{i\to \infty} \frac{u_{\partial_n G^{Y_i}}(X) - H^{Y_i}(X)}{G^{Y_i}(X)} \geq -c \]
Since the bound is for all $c>0$, we conclude that 
\[ \frac{\partial_n G(X)}{G(X)} \geq 0\]
as desired.

\bigskip

 {\bf Step 2: Show that $\partial_n G(X) \geq cG(X) / \delta(X)$ for $X\in\om_0$.} From Step 1, we get that for any Lipschitz graph domain $\Omega_0$ and any Green function with pole at infinity $G$ on $\Omega_0$, we have 
 \begin{equation} \label{dnG>0}
 \partial_n G > 0.
 \end{equation}
 Indeed, $\partial_n G$ is a non-zero non-negative harmonic function in $\Omega_0$, so the Harnack inequality forces $\partial_n G$ to be positive.

 We want to prove that there exists $c>0$ (depending only on $n$ and the Lipschitz constant of $g$) such that 
  \begin{equation} \label{dnG>cG}
  \partial_n G(X) \geq cG(X) / \delta(X) \qquad \text{ for all } X\in \Omega_0.
  \end{equation}
  This claim easily yields \eqref{H1}--\eqref{H2}.

 To show \eqref{dnG>cG}, it suffices to show that there is $c>0$ depending only on $n$ and the Lipschitz constant of $g$, such that $\dr_nG(0,1)>cG(0,1)$ (recall that $0\in\pom_0$). We can assume without loss of generality that $G(0,1)=1$. By the Harnack inequality, $0<G(X)\lesssim 1$ for all  $X \in B(0,2)\cap\om_0$, with the implicit constant depending only on the dimension and the Lipschitz constant.  Moreover, since $G$ is smooth (and in particular, continuous) in $B(0,1)\cap\om_0$, there exists $c_0 >0$ (depending only on $n$ and the Lipschitz constant of $g$) such that $G(0,c_0) \le 1/2$. So we have that 
\[
\frac12\le G(0,1) - G(0,c_0) = \int_{c_0}^1 \partial_n G(0,s) ds.
\]
By the mean value theorem, there exists $s_0 \in (c_0,1)$ such that $\partial_n G(0,s_0) \ge 1/2$. Finally, since $\partial_n G$ is harmonic and nonnegative, the Harnack inequality implies that $\partial_n G(0,1) \ge c$ for some $c>0$.

\section{Solvability of the $L^p$ Dirichlet problem under $C^{1,\alpha}$ diffeomorphism of convex domains}\label{S.App}
In this appendix, we show by a simple method that if the boundary of a bounded domain $\om$ (or an unbounded Lipschitz graph domain) is locally (or globally, in the unbounded case)``a $C^{1,\alpha}$ diffeomorphism of a convex gragh'', then the $L^p$ Dirichlet problem $(D)_p$ is solvable for all $1<p<\infty$ for the Laplacian in $\om$. Particularly, this generalizes the results in convex and semiconvex domains for the Laplacian \cite{MMY10}. 
Recall that a semiconvex domain is a Lipschitz domain satisfying the exterior ball condition, which is equivalent to say that the boundary is locally the graph of the sum of a convex and a $C^{1,1}$ function; see \cite[Proposition 3.6]{MMY10}.

\begin{lemma}\label{lem-LipEst}
    Let $\Omega = \Phi(\{(x,t): t> \phi(x) \})$, where $\phi:\R^{n-1}\to \R$ is a convex function with Lipschitz constant $M$, and $\Phi:\Rn\to\Rn$ is a $C^{1,\alpha}$ diffeomorphism. Assume $\phi(0) = 0$ (i.e., $0\in \partial \Omega$). Let $L=-\diver A\nabla$ be an elliptic operator with coefficients in $C^\beta(\om)$ for some $\beta>0$. Let $u$ be a weak solution to $Lu=0$ in $\Omega \cap B_1(0)$ with $u=0$ on $\pom\cap B_1(0)$. Then
    \begin{equation}\label{est.DuInConvex}
        \| \nabla u \|_{L^\infty(\Omega\cap B_{1/2})} \le C \| u \|_{L^\infty(\Omega \cap B_1)},
    \end{equation}
    where $C>0$ depends on $n$, $M$, the ellipticity constant, $\norm{\Phi}_{C^{1,\alpha}(B_1)}$,  and $\norm{A}_{C^\beta(\om\cap B_1)}$. 
\end{lemma}
\begin{proof}
    Let $v = u\circ \Phi$. Then $v$ satisfies $L_\Phi v = -\diver(A_\Phi \nabla v) = 0$ in $\Phi^{-1}(\Omega \cap B_1)$, where
    \begin{equation*}
        A_\Phi = |\det \nabla\Phi| (\nabla\Phi)^{-T} (A\circ \Phi) (\nabla\Phi)^{-1}.
    \end{equation*}
     Clearly, $A_{\Phi}$ also satisfies the uniform ellipticity condition and $C^{\alpha_1}$ continuous (since $\Phi$ is a $C^{1,\alpha}$ diffeomorphism), with $\alpha_1 = \min \{ \alpha, \beta \} > 0$. Hence, the desired estimate \eqref{est.DuInConvex} is equivalent to
    \begin{equation}\label{est.DvInConvex}
        \| \nabla v \|_{L^\infty(\Phi^{-1}(\Omega\cap B_{1/2}))} \le C \| v \|_{L^\infty(\Phi^{-1}(\Omega \cap B_1))}.
    \end{equation}
    But now $\Phi^{-1}(\Omega \cap B_1) = \{ (x,t): t > \phi(x) \} \cap \Phi^{-1}(B_1)$. By assumption, the part $\{ (x,t): t = \phi(x) \} \cap \Phi^{-1}(B_1)$ is convex. The gradient estimate \eqref{est.DvInConvex} over convex boundaries is more or less well-known; see \cite[Chapter 14.2]{GT83} for elliptic operator in nondivergence form. We provide a proof for our elliptic operator in divergence form for the reader's convenience.
    
    Up to a translation and rotation, we can assume  that $\Phi(0) = 0$, $\phi(0)=0$, and by the convexity of $\phi$, that $\phi(x) \ge 0$ for $x\in\R^{n-1}$. 
    We claim: for $t > 0$ and $(0,t) \in \Phi^{-1}(B_{3/4}) \subset \subset \Phi^{-1}(B_1)$, we have
    \begin{equation}\label{est.v.lineardecay}
        |v(0,t)| \le Ct\| v \|_{L^\infty(\Phi^{-1}(\Omega\cap B_1))}
    \end{equation}
    This is proved by the comparison principle. We construct a barrier function in the following way. Let $M = \| v \|_{L^\infty(\Phi^{-1}(\Omega\cap B_1))}$. Let $w$ be a barrier function satisfying $L_\Phi(w) = 0$ in $\{ t > 0 \} \cap \Phi^{-1}(B_1)$, $w = 0$ on $\{ t = 0\} \cap \Phi^{-1}(B_1)$ and $w = M$ on $\{ t> 0\} \cap \partial \Phi^{-1}(B_1)$. Clearly, $0\le w \le M$ in $\{ t > 0\} \cap \Phi^{-1}(\Phi)$. The key fact we will use from the convexity of $\phi(x)$ is that $\{ t> \phi(x)\} \cap \Phi^{-1}(B_1) \subset \{ t > 0\} \cap \Phi^{-1}(B_1)$ (due to the fact $\phi(x) \ge 0$). Therefore, we have $|v| = 0 \le w$ on the convex part of the boundary $\{ t = \phi(x) \} \cap \Phi^{-1}(B_1)$ and $|v| \le M = w$ on the rest of the boundary $\{ t > \phi(x) \} \cap \partial \Phi^{-1}(B_1)$. By the comparison principle, we have $|v| \le w$ in $\{ t > \phi(x) \} \cap \Phi^{-1}(B_1)$. By the Schauder estimate of $w$ over a flat boundary $\{ t = 0 \}$, we have $|w(x,t)| \le CtM$ for any $(x,t) \in \{t > 0 \} \cap \Phi^{-1}(B_{3/4})$. This particularly implies \eqref{est.v.lineardecay}.

    Finally, by applying the previous argument for \eqref{est.v.lineardecay} to all the boundary points on $\{ t = \phi(x)\} \cap \Phi^{-1}(B_{3/4})$, we actually have $|v(X)| \le C\delta(X) M$, where $\delta(X)$ is the distance from $X$ to the boundary $\{ t = \phi(x)\}$. Hence, by the interior Schauder estimate, for any $X \in \{ t > \phi(x)\} \cap \Phi^{-1}(B_{1/2})$, we have
    \begin{equation*}
        |\nabla v(X)| \le C \delta(X)^{-1} \| v \|_{L^\infty(B(X, c\delta(X)))} \le CM.
    \end{equation*}
    This ends the proof of \eqref{est.DvInConvex}.
\end{proof}

\begin{remark}
    The previous lemma (and thus the next theorem) is also valid if $\Phi$ is a $C^{1,{\rm Dini}}$-type diffeomorphism, thanks to the well-known gradient estimate for elliptic operators with Dini continuous coefficients; see \cite{Lieb86}. However, the argument does not work if $\Phi$ is merely $C^1$, due to the classical failure of the boundedness of gradient for the elliptic operator $L_\Phi = -\diver (A_\Phi \nabla )$ with continuous coefficients. 
  \end{remark}

\begin{theorem}\label{thm-Dp-C1a} 
Let $\om$ be either a bounded Lipschitz domain, or an unbounded Lipschitz graph domain in a form of $\om=\set{(x,t)\in\R^{n-1}\times\R: t> \phi_0(x) + \phi_1(x)}$, where $\phi_0$ is a convex function and $\|\nabla \phi_0 \|_{L^\infty(\R^{n-1})} + \| \phi_1 \|_{L^\infty(\R^{n-1})} + \| \nabla \phi_1 \|_{L^\infty(\R^{n-1})} \le M$. Suppose that there exists $r_0>0$ such that for any $x_0\in \partial \Omega$, there exists a coordinate basis $\{e_1,\dots,e_n\}$, an $M$-Lipschitz convex function $g_{x_0}$ and a $C^{1,\alpha}$ diffeomporphism $\Phi_{x_0}:\, \R^n \to \R^n$ with $|\nabla\Phi_{x_0}| + | \nabla\Phi_{x_0}^{-1} | \le M$ such that 
\[\Omega \cap B(x_0,r_0) = \Phi_{x_0}\Big( \{ (x,t) \in \R^{n-1}\times\R,\, t>g_{x_0}(x)\} \Big) \cap B(x_0,r_0).\]
Then $(D)_p$ for $-\Delta$  is solvable for all $1<p<\infty$ in $\Omega$.
\end{theorem}

\begin{proof}
\underline{Case 1:} when $\om$ is an unbounded Lipschitz graph domain. It suffices to prove that for any $q>1$, there exists a constant $C>0$ such that 
\begin{equation}\label{eq.ubddRH}
    \br{\frac{1}{\sigma(\Delta(X,r))}\int_{\Delta(X,r)}k(y)^qd\sigma(y)}^{1/q}\le  \frac{C}{\sigma(\Delta(X,r))}\int_{\Delta(X,r)}k(y)d\sigma(y)=C\frac{\omega(\Delta(X,r))}{\sigma(X,r)}
\end{equation}
for all $X\in\pom$, $r>0$, where $k(y)=\frac{d\omega}{d\sigma}(y)$ and $\omega$ is the harmonic measure with pole at infinity (see Lemma~\ref{lem GrEm_infty}).  

Since $\phi$ is Lipschitz with constant at most $M$, for any $y=(y',\phi(y'))\in \Delta(X,r)$, we have 
\[
k(y)=\lim_{s\to0}\frac{\omega(\Delta(y,s))}{\sigma(\Delta(y,s))}\lesssim \limsup_{s\to 0}\frac{G(y',s+\phi(y'))}{s}
\]
by Lemma~\ref{lemG=om}, where $G$ is the Green function for $-\Delta$ with pole at infinity. Apply Lemma \ref{lem-LipEst} to $u=G$ in $B(X,r_0)\cap\om$, we get that (for $s>0$ sufficiently small)
\begin{equation}\label{eq.fullscaleLip}
\frac{G(y',s+\phi(y'))}{s}\le \frac{C\norm{G}_{L^\infty(B(X,r)\cap\om)}}{r},
\end{equation}
for any $0<s<r\le r_0$, where the constant $C$ is independent of $X$ and $r$. To extend the same estimate for $r>r_0$, we need a large-scale Lipschitz estimate, which has been established in  \cite[Theorem 4.2]{Z21} if the boundary graph is a ``convex + bounded'' function (i.e., $\partial \Omega = \{ t = \phi_0(x) + \phi_1(x) \}$), and in our case can be formed as
\[
\norm{G}_{L^\infty(B(X,r_0)\cap\om)} \le \frac{C \norm{G}_{L^\infty(B(X,r)\cap\om)} }{r}
\]
for all $r>r_0$. Combing the local estimate with the large-scale Lipschitz estimate, we obtain \eqref{eq.fullscaleLip} for all $r > s > 0$.
By the Bourgain estimate (see e.g. \cite[Lemma 4.4]{DK82}), one has $\norm{G}_{L^\infty(B(X,2r)\cap\om)}\le G(A_{X,r})$, where $A_{X,r}$ is a corkscrew point relative to $X$ and $r$, that is, $A_{X,r}\in B(X,Mr)\cap\om$ satisfies $\dist(A_{X,r},\pom)\approx |A_{X,r}-X|\approx r$. By Lemma~\ref{lemG=om} again, we have that $\frac{G(A_{X,r})}{r}\approx \frac{\omega(\Delta(X,r))}{\sigma(X,r)}$. Putting everything together, one has 
\begin{equation*}
     \br{\frac{1}{\sigma(\Delta(X,r))}\int_{\Delta(X,r)}k(y)^qd\sigma(y)}^{1/q}\le \frac{C\norm{G}_{L^\infty(B(X,2r)\cap\om)}}{r}\le C\frac{\omega(\Delta(X,r))}{\sigma(X,r)}
\end{equation*}
as desired.

\underline{Case 2:} when $\om$ is bounded. It suffices to show that for any $q>1$, there exists a constant $C>0$ such that 
\begin{equation}\label{eq.bddRH}
    \br{\frac{1}{\sigma(\Delta(X,r))}\int_{\Delta(X,r)}k^{X_0}(y)^qd\sigma(y)}^{1/q}\le  \frac{C}{\sigma(\Delta(X,r))}\int_{\Delta(X,r)}k^{X_0}(y)d\sigma(y)=C\frac{\omega^{X_0}(\Delta(X,r))}{\sigma(X,r)}
\end{equation}
for all $X\in\pom$, $0<r<\diam(\om)$, where $X_0\in\om$ is a point satisfying $\dist(X_0,\pom)\approx\diam(\om)$, $k^{X_0}(y)=\frac{d\omega^{X_0}}{d\sigma}(y)$, and $\omega^{X_0}$ is the harmonic measure with pole $X_0$. By definition, there is $0<r_0<\diam(\om)$ such that for any $X\in\pom$, \[\om\cap B(X,r_0)=\Phi_{X}(\set{(x,t)\in B(X,r_0): t>g_X(x)} )\] where $g_X$ is convex with Lipschitz constant at most $M$. By a covering argument, it suffices to show \eqref{eq.bddRH} for $0<r<\min\set{r_0/4,\dist(X_0,\pom)/4}$. Since for any $X\in\pom$ and $0<r<\min\set{r_0/4,\dist(X_0,\pom)/4}$, $G(\cdot, X_0)$ is a solution to $-\Delta u=0$ in $B(X,2r)\cap\om$ with $G(\cdot, X_0)=0$ on $\pom$, one can apply \eqref{lem-LipEst} to $G(\cdot, X_0)$ in $B(X,r)\cap\om$ and the rest of the argument goes as in Case 1.  
\end{proof}

\begin{remark}
    In the case of an unbounded Lipschitz graph domain considered in Theorem \ref{thm-Dp-C1a}, if $\phi_1 = 0$, then $\Omega$ is convex. Then $\Phi_{x_0}$ can be taken as identity transform and $g_{x_0}(x) = \phi_0(x)$ for any $x_0\in \partial \Omega$.
\end{remark}

\begin{theorem}
    Under the same assumptions as Theorem \eqref{thm-Dp-C1a}, the $L^p$ Regularity problem $(R)_p$ for the Laplace operator is solvable  in $\om$ for any $1<p<\infty$.
\end{theorem}

\bibliographystyle{alpha}
\bibliography{reference}

\end{document}